\documentclass[11pt]{article}

\usepackage[english]{babel}
\usepackage{amsmath,amsfonts,amssymb,amsthm}
\usepackage{latexsym}
\usepackage{makeidx}
\usepackage{amscd}
\usepackage{hyperref}

\usepackage{vmargin}
\setmarginsrb{0.65in}{1in}{0.65in}{1in}%
             {0.1in}{0.1in}{0.1in}{0.3in}

\numberwithin{equation}{section}

\newtheorem{definition}{Definition}[section]
\newtheorem{lemma}[definition]{Lemma}
\newtheorem{theorem}[definition]{Theorem}
\newtheorem{corollary}[definition]{Corollary}

\newtheorem{proposition}[definition]{Proposition}
\newtheorem{fact}[definition]{Fact}
\newtheorem{em-example}[definition]{Example}
\newtheorem{em-def}[definition]{Definition}        
\newtheorem{em-remark}[definition]{Remark}         
\newtheorem{em-question}[definition]{Question}

\newtheorem{problem}[definition]{Problem}

\newenvironment{example}{\begin{em-example} \em }{ \end{em-example}}
\newenvironment{remark}{\begin{em-remark} \em }{\end{em-remark}}

\newcommand{\R}{\mathbb R}
\newcommand{\N}{\mathbb N}
\newcommand{\C}{\mathbb C}
\newcommand{\Q}{\mathbb Q}
\newcommand{\Z}{\mathbb Z}

\newcommand{\f}{\phi}

\def\abg{\mathbf{AbGrp}}
\def\Mod{\mathbf{Mod}}
\def\id{\mathrm{id}}
\def\af{\mathbf{Flow}}

\DeclareMathOperator{\End}{End}
\DeclareMathOperator{\Aut}{Aut}

\def\htop{h_{\mathrm{top}}}
\DeclareMathOperator{\ent}{ent}
\DeclareMathOperator{\AT}{AT}

\global\def\hull#1{\left\langle{#1}\right\rangle}

\input xy
\xyoption{all}

\title{Entropy on abelian groups \footnote{MSC2010: 
20K30; 
37A35; 
37B40; 
11R06. 
Keywords: discrete dynamical system, algebraic entropy, group endomorphism, abelian group, addition theorem, Mahler measure, Lehmer problem.}}

\author{Dikran Dikranjan
\thanks{The first named author gratefully acknowledges the FY2013 Long-term visitor grant~L13710 by the Japan Society for the Promotion of Science (JSPS).}
\\{\footnotesize {\tt  dikran.dikranjan@uniud.it}} 
\\{\footnotesize Dipartimento di Matematica e Informatica,}
\\{\footnotesize Universit\`{a} di Udine,}
\\{\footnotesize Via delle Scienze, 206 - 33100 Udine, Italy} 
 \and Anna Giordano Bruno 
 \thanks{The second named author is supported by Programma SIR 2014 by MIUR (project GADYGR, number RBSI14V2LI, cup G22I15000160008).}
 \thanks{Work partially supported by the Fondazione Cassa di Risparmio di Padova e Rovigo (Progetto di Eccellenza ``Algebraic structures and their applications'') and by the ``National Group for Algebraic and Geometric Structures, and Their Applications'' (GNSAGA - INdAM).}
\\{\footnotesize {\tt  anna.giordanobruno@uniud.it}} 
\\{\footnotesize Dipartimento di Matematica e Informatica,}
\\{\footnotesize Universit\`{a} di Udine,}
\\{\footnotesize Via delle Scienze, 206 - 33100 Udine, Italy}
 }

\date{Dedicated to Professor Justin Peters}

\begin{document}

\maketitle


\abstract{We introduce the algebraic entropy for endomorphisms of arbitrary abelian groups, appropriately modifying existing notions of entropy.
The basic properties of the algebraic entropy are given, as well as various examples. The main result of this paper is the Addition Theorem showing that the algebraic entropy is additive in appropriate sense with respect to invariant subgroups. We give several applications of the Addition Theorem, among them the Uniqueness Theorem for the algebraic entropy in the category of all abelian groups and their endomorphisms. Furthermore, we point out the delicate connection of the algebraic entropy with the Mahler measure and Lehmer Problem in Number Theory.}

\section{Introduction}

Inspired by the notion of entropy invented by Clausius in Thermodynamics in the fifties of the nineteenth century, Shannon introduced entropy in Information Theory by the end of the forties of the last century. A couple of years later, Kolmogorov in \cite{K} and Sinai in \cite{Sinai} introduced the measure entropy in Ergodic Theory. By an appropriate modification of their definition, Adler, Konheim and McAndrew in \cite{AKM} obtained the topological entropy $\htop$ for continuous self-maps of compact topological spaces.
Moreover, they briefly sketched an  idea on how to define an algebraic entropy for endomorphisms of discrete abelian groups. Weiss in \cite{W} developed this idea and studied the algebraic entropy for endomorphisms of torsion abelian groups. Later on, Peters in \cite{Pet} modified this notion of algebraic entropy for automorphisms of arbitrary abelian groups.  The interest in algebraic entropy was renewed after the recent article \cite{DGSZ}, where 
 the case of endomorphisms of torsion abelian groups is thoroughly studied. 

\medskip
We introduce some notation in order to discuss in detail the above mentioned notions of algebraic entropy, with the aim to extend Peters' definition to arbitrary endomorphisms of abelian groups.
 
Let $G$ be an abelian group and let $\phi\in \End(G)$. For a non-empty subset $F$ of $G$ and $n\in\N_+$, the \emph{$n$-th $\phi$-trajectory of $F$} is 
$$T_n(\phi,F)=F+\phi(F)+\ldots+\phi^{n-1}(F),$$ 
and the \emph{$\phi$-trajectory of $F$} is $$T(\phi,F)=\sum_{n\in\N}\phi^n(F).$$

For a finite subgroup $F$ of $G$, the limit 
\begin{equation*}
H(\phi,F)={\lim_{n\to\infty}\frac{\log|T_n(\phi,F)|}{n}}
\end{equation*}
exists (see Fact \ref{Fact:limit:free}) and it is the \emph{algebraic entropy of $\phi$ with respect to $F$}. According to \cite{AKM,W}, the \emph{algebraic entropy} $\ent$ of $\phi$ is defined as
\begin{equation*}
\ent(\phi)=\sup\{H(\phi,F): F\ \text{finite subgroup of}\ G\}.
\end{equation*}
Clearly, this notion is more suitable for endomorphisms of \emph{torsion} abelian groups, as $\ent(\phi)=\ent(\phi\restriction_{t(G)})$, where $t(G)$ is the torsion part of $G$. We refer to \cite{DGSZ} as a convenient reference for the algebraic entropy $\ent$. 

\smallskip
Peters introduced a different notion of algebraic entropy for {\em automorphisms} $\phi$ of abelian groups, using instead of $T_n(\phi,F)=F+\phi(F)+\ldots+\phi^{n-1}(F)$, the ``negative part of the trajectory''
 $$
 T_n^-(\phi,F)=F+\phi^{-1}(F)+\ldots+\phi^{-n+1}(F) = T_n(\phi^{-1},F),
 $$ 
and instead of finite subgroups $F$ of $G$ just non-empty finite \emph{subsets} $F$ of $G$. Therefore, Peters' algebraic entropy may have non-zero values also for torsion-free abelian groups.

Obviously, this definition can be given using $T_n(\phi,F)$ as well, since  $T_n^-(\phi,F)=T_n(\phi^{-1},F)$. The approach with the ``positive'' partial trajectories $T_n(\phi,F)$ has the advantage to be applicable to any endomorphism $\phi$ (whereas $T_n^-(\phi,F)$ may be infinite when $\phi$ is not injective). Hence, we define the algebraic entropy $h$ for endomorphisms $\phi$ of abelian groups $G$ as follows.

For a non-empty finite \emph{subset} $F$ of $G$, the limit 
\begin{equation}\label{lim-eq}
H(\phi,F)=\lim_{n\to\infty}\frac{\log|T_n(\phi,F)|}{n}
\end{equation}
still exists (see Lemma \ref{lim-ex}(b)), and we call it the \emph{algebraic entropy of $\phi$ with respect to $F$}. The \emph{algebraic entropy} of $\phi$ is 
$$h(\phi)=\sup\{H(\phi,F):F\in[G]^{<\omega}\},$$
where $[G]^{<\omega}$ denotes the family of all non-empty finite subsets of $G$.

For endomorphisms of torsion abelian groups, $h$ coincides with $\ent$. More precisely, if $G$ is an abelian group and $\phi\in\End(G)$, then $\ent(\phi)=\ent(\phi\restriction_{t(G)})=h(\phi\restriction_{t(G)})$. 

\smallskip
The definition of $h$ can be extended to the non-abelian case. However, here we prefer to keep the exposition in the abelian setting, due to the differences and  the technical difficulties that naturally appear considering non-abelian groups. The reader interested  in the non-abelian case and in the connection to the growth rate of finitely generated groups may see \cite{DG_Pak,DG_PC}. For other related algebraic entropy functions see \cite{Bow,DGS,DGSV,FFK}.

\bigskip
The main result of this paper is the following key additivity property of the algebraic entropy, called Addition Theorem. The torsion case of this theorem is covered by the Addition Theorem for $\ent$ proved in \cite{DGSZ}.

\begin{theorem}[\underline{Addition Theorem}]\label{AT}
Let $G$ be an abelian group, $\phi\in \End(G)$, $H$ a $\phi$-invariant subgroup of $G$ and $\overline\phi\in \End(G/H)$ the endomorphism induced by $\phi$. Then 
\begin{equation}\label{AT-f}
h(\f)=h( \f\restriction_H) + h(\overline \f).
\end{equation}
\end{theorem}

In terms of the following diagram describing the above situation
$$\xymatrix{
H \ar[r] \ar[d]^{\phi\restriction_H} & G \ar[r] \ar[d]^\phi &G/H \ar[d]^{\overline\phi} \\
H \ar[r] & G \ar[r] & G/H
}$$
the equality \eqref{AT-f} of the Addition Theorem reads: the algebraic entropy of the mid endomorphism is the sum of the algebraic entropies of the two side ones. 
We adopt the notation $$\AT_h(G,\phi,H)$$ for an abelian group $G$, $\f\in\End(G)$ and a $\phi$-invariant subgroup $H$ of $G$,  to indicate briefly that ``the equality \eqref{AT-f} holds for the triple $(G,\phi,H)$''. 

\medskip
A relevant part of the proof of the Addition Theorem is based on the so called Algebraic Yuzvinski Formula (see Theorem \ref{Yuz} below), recently proved in \cite{GV}. To state the Algebraic Yuzvinski Formula, we need to recall the notion of Mahler measure, playing an important role in Number Theory and Arithmetic Geometry (see \cite[Chapter 1]{Ward0} and \cite{Hi}). 

For a primitive polynomial $f(t)=a_0+a_1t+\ldots+a_kt^k\in\Z[t]$, let $\alpha_1,\ldots,\alpha_k\in \mathbb C$ be the roots of $f(t)$ taken with their multiplicity.
The \emph{Mahler measure} of $f(t)$ is 
$$
m(f(t))= \log|a_k| + \sum_{|\alpha_i|>1}\log |\alpha_i|.
$$

The Mahler measure of a linear transformation $\f$ of a finite dimensional rational vector space $\Q^n$, $n\in\N_+$, is defined as follows. 
Let $g(t)\in\Q[t]$ be the characteristic polynomial of $\phi$. Then there exists a smallest $s\in\N_+$ such that $sg(t)\in\Z[t]$ (so $sg(t)$ is primitive). The \emph{Mahler measure of $\phi$} is $m(\phi)=m(sg(t))$.

\begin{theorem}[Algebraic Yuzvinski Formula]\label{Yuz}  
For $n\in \N_+$ and $\phi\in\End(\Q^n)$,
$$h(\f)=m(\phi).$$
\end{theorem}

The counterpart of this formula for the topological entropy was proved by Yuzvinski in \cite{juz67} (see also \cite{LW}). A direct proof (independent from the topological case) of the Algebraic Yuzvinski Formula was recently given in \cite{GV}, extending results from \cite{Vi} and using the properties of the Haar measure on locally compact abelian groups. A weaker form for the case of zero algebraic entropy is proved in \cite{DGZ} using basic linear algebra.

\medskip
The Algebraic Yuzvinski Formula gives one of the two ``normalization axioms'' that entail uniqueness of $h$; in fact, Theorem \ref{UT}(e) concerns the values $h(\phi)=m(\phi)$, where $\phi\in\End(\Q^n)$ and $n\in\N_+$. The other one is furnished by the following family of group endomorphisms. For any abelian group $K$ the \emph{(right) Bernoulli shift} $\beta_{K}:K^{(\N)}\to K^{(\N)}$ is defined by $(x_0,x_1,x_2,\ldots)\mapsto(0,x_0,x_1,\ldots)$. Then $h(\beta_K)=\log|K|$, with the usual convention that $\log|K|=\infty$ if $K$ is infinite (see Example \ref{beta}).  The value $h(\beta_K)$, when $K$ runs over all finite abelian groups, is the condition in Theorem \ref{UT}(d).
 
\medskip
 The Uniqueness Theorem for the algebraic entropy is our second main result, its proof uses the Addition Theorem and the Algebraic Yuzvinski Formula (see Theorem \ref{UT} below, where the algebraic entropy $h$ is considered as a collection $h=\{h_G:G\ \text{abelian group}\}$ of functions $h_G:\End(G)\to\R_{\geq0}\cup\{\infty\}$ defined by $h_G(\phi)=h(\phi)$ for every $\phi\in\End(G)$). It extends the Uniqueness Theorem given in \cite{DGSZ} for $\ent$ in the class of torsion abelian groups, that was inspired, in turn, by the Uniqueness Theorem for the topological entropy of Stojanov (see \cite{St}).

\smallskip
If $G$ and $H$ are abelian groups, $\phi\in\End(G)$ and $\eta\in\End(H)$, we say that $\phi$ and $\eta$ are \emph{conjugated} if there exists an isomorphism $\xi:G\to H$ such that $\eta=\xi\phi\xi^{-1}$. Sometimes we specify also the isomorphism $\xi$ by saying that $\phi$ and $\eta$ are \emph{conjugated by $\xi$}.

\begin{theorem}[\underline{Uniqueness Theorem}]\label{UT}
The algebraic entropy $h$ of the endomorphisms of abelian groups is the unique collection $h=\{h_G:G\ \text{abelian group}\}$ of functions $h_G:\End(G)\to\R_{\geq0}\cup\{\infty\}$ such that:
\begin{itemize}
\item[(a)] if $\phi\in\End(G)$ and $\eta\in\End(H)$ are conjugated, then $h_G(\f)=h_H(\eta)$ (i.e., $h$ is invariant under conjugation);
\item[(b)] if $G$ is an abelian group, $\phi \in \End(G)$ and $G$ is a direct limit of $\phi$-invariant subgroups $\{G_i:i\in I\}$, then $h_G(\phi) = \sup_{i\in I}\!h_{G_i}(\phi\restriction_{G_i})$; 
\item[(c)] $\AT_{h_G}(G,\phi,H)$ for every abelian group $G$, every $\phi\in\End(G)$ and every $\phi$-invariant subgroup $H$ of $G$;
\item[(d)] $h_{K^{(\N)}}(\beta_K)=\log|K|$ for any finite abelian group $K$;
\item[(e)] $h_{\Q^n}(\phi)=m(\phi)$ for every $n\in\N_+$ and every $\phi\in \End(\Q^n)$.
\end{itemize}
\end{theorem}

\medskip
Another application of the Addition Theorem in this paper concerns Lehmer Problem from Number Theory (see Problem \ref{h>0-pb}), posed by Lehmer in 1933, and still open (see \cite{Ward0,Hi,MRQ,Sm}). Let 
$$\mathfrak L:=\inf\{m(f(t))>0:f(t)\in\Z[t]\ \text{primitive}\}$$
(obviously, $\mathfrak L= \inf\{m(f(t))>0:f(t)\in\Z[t]\ \text{monic}\}$).

\begin{problem}[Lehmer Problem]\label{h>0-pb}
Is $\mathfrak L>0$?
\end{problem}

Let us see the connection of Lehmer Problem to algebraic entropy. Note first that the finite values of the algebraic entropy $\ent$ belong to the set $\log\N_+:=\{\log n:n\in\N_+\}$ (see Fact \ref{Fact:limit:free}). Then 
\begin{equation}\label{ent>log2}
\inf\{\ent(\phi)>0:G\ \text{abelian group}, \phi\in\End(G)\}=\log2.
\end{equation}
Consider now the infimum of the positive values of the algebraic entropy $h$:
$$\varepsilon:=\inf\{h(\phi)>0:G \text{ abelian group},\ \phi\in\End(G)\}.$$
The next theorem shows that the problem of deciding whether $\varepsilon>0$  is equivalent to Lehmer Problem. The counterpart of this theorem for the topological entropy is apparently well-known.
 
\begin{theorem}\label{epsilon}
The  equalities \ $\varepsilon=\inf\{h(\phi)>0:n\in\N_+,\phi\in\Aut(\Q^n)\}=\mathfrak L$\ hold.
\end{theorem}

Finally, we see that a positive solution to Lehmer Problem is equivalent to the realization of the finite values of the algebraic entropy, in the following sense.

\begin{definition}
Let $G$ be an abelian group and $\phi\in\End(G)$ with $h(\phi)<\infty$. We say that $h(\phi)$ \emph{realizes} if there exists $F\in[G]^{<\omega}$ such that $h(\phi)=H(\phi,F)$. 
\end{definition}

\begin{theorem}[\underline{Realization Theorem}]\label{RT}
The following conditions are equivalent:
\begin{itemize}
\item[(a)] $\mathfrak L>0$;
\item[(b)] $h(\phi)$ realizes for every abelian group $G$ and every $\phi\in\End(G)$ with $h(\phi)<\infty$.
\end{itemize}
\end{theorem}

\medskip
An application of the Addition Theorem and the Uniqueness Theorem is the so called Bridge Theorem, given in \cite{DG2}. In more detail, both Weiss in {\cite{W}} and Peters in \cite{Pet} connected the algebraic entropy to the topological entropy by using Pontryagin duality. 
For an abelian group $G$ and $\phi\in \End(G)$, we denote by $\widehat G$ the Pontryagin dual of $G$ and by $\widehat\phi:\widehat G\to \widehat G$ the dual endomorphism of $\phi$ (see \cite{HR,P}). It is known that $\widehat G$ is compact, and it is also metrizable when $G$ is countable. 

\begin{theorem}[Bridge Theorem]\label{Weiss}\label{Peters}
Let $G$ be an abelian group and $\phi\in \End(G)$.
\begin{itemize}
\item[(a)] \emph{\cite{W}} If $G$ is torsion, then $\ent(\phi)=\htop(\widehat\phi)$.
\item[(b)] \emph{\cite{Pet}} If $G$ is countable and $\phi$ is an automorphism, then $h(\phi)=\htop(\widehat\phi)$.
\end{itemize}
\end{theorem}

It was proved in \cite{DG2} that the conclusion $h(\phi)=\htop(\widehat\phi)$ remains true for \emph{arbitrary} abelian groups and their \emph{endomorphisms} $\phi$. Two proofs of this result are given in \cite{DG2}, one is based on the Addition Theorem, the Algebraic Yuzvinski Formula and their counterparts for the topological entropy, the second uses the Uniqueness Theorem.

For another extension of Theorem \ref{Weiss}(a) to the case of (totally disconnected) locally compact abelian groups, see \cite{DG-BT}. 
We do not discuss here this  extension as well as similar results from \cite{Pet1,Vi1}, since this paper is limited to the discrete case; various definitions of algebraic entropy in the general case of locally compact abelian groups can be found in \cite{Pet1,Vi,Vi1}.

\medskip
Let us mention, finally, another application of the Addition Theorem given in \cite{DG}, connecting the algebraic entropy with the growth rate of abelian groups with respect to endomorphisms. 
For further details on this and other applications of the Addition Theorem and the Algebraic Yuzvinski Formula refer to \cite{DG_Pak,DG_PC,DSV,GV-app} (see also the survey \cite{DSV}).

\bigskip
The paper is organized as follows.

\smallskip
Section \ref{sec-def} contains first of all the proof of the existence of the limit in \eqref{lim-eq} defining the algebraic entropy $h$. Then, the basic properties of $h$ are listed and verified (see Lemma \ref{restriction_quotient} and Proposition \ref{properties}); they are counterpart of the properties of $\ent$ proved in \cite{DGSZ,W}, as well as of the known properties of the topological entropy. Moreover, basic examples are given, among them the Bernoulli shift.

\smallskip
In Section \ref{tf-sec} we consider various properties of the algebraic entropy of endomorphisms $\phi$ of torsion-free abelian groups $G$. 
We underline the preservation of the algebraic entropy under the extension of an endomorphism $\phi$ of $G$ to the divisible hull of $G$ (see Proposition \ref{AA}). 
As examples, the algebraic entropy of the endomorphisms of $\Z$ and $\Q$ is computed (see Examples \ref{Zex} and \ref{Qex}).

\smallskip
The aim of Section \ref{fg-sec} is to reduce the computation of the algebraic entropy to the case of appropriately small groups.  We see in a first reduction that this means endomorphisms of \emph{countable} abelian groups (see Lemma \ref{reduction-to-countable}).
Moreover, we consider the structure of $\Z[t]$-module induced on an abelian group $G$ by the action of its endomorphism $\phi$, and we denote by $G_\phi$ the group $G$ considered as a $\Z[t]$-module under the action of $\phi$. One can compute the algebraic entropy of $\phi$ as the supremum of the algebraic entropies of the restrictions of $\phi$ to the finitely generated $\Z[t]$-submodules of $G_\phi$ (see Lemma \ref{remark-G_F-gen}). In particular, Proposition \ref{sup->lim} reduces the proof of the Addition Theorem to the case when $G_\phi$ is a finitely generated $\Z[t]$-module and $G$ has finite rank.

\smallskip
Section \ref{AT-sec} is dedicated to the Addition Theorem. Its proof is exposed in several steps, which are partial reductions gradually restricting the class of groups. We start by showing $\AT_h(G,\phi,t(G))$ (see Proposition \ref{AT-t(G)}). This allows us to focus on the case of torsion-free abelian groups. Using the properties established in the previous sections, we reduce to endomorphisms of finite-rank divisible torsion-free abelian groups, i.e., to endomorphisms $\phi:\Q^n\to \Q^n$ for some $n\in\N_+$. The next step (see Proposition \ref{red-to-inj}) is to see that $\f$ can be supposed injective (hence, also surjective). The ultimate case of automorphisms of $\Q^n$ can be managed through the Algebraic Yuzvinski Formula (see Theorem \ref{AT-tf-pure}). 

\smallskip
In Section \ref{uniq-sec} we prove the Uniqueness Theorem.
We give a direct proof of this result, but we show also how it can be deduced from a theorem by V\'amos on length functions (see Definition \ref{L-def} below, and \cite{NR,V}).
The approach based on length functions is powerful and clarifies the role of the five axioms in the statement of the Uniqueness Theorem. Indeed, the first three axioms witness the fact that the algebraic entropy is a length function of the category $\Mod_{\Z[t]}$ of all $\Z[t]$-modules, while the fourth and the fifth are used as normalizations. More precisely, a length function of $\Mod_{\Z[t]}$ is determined by its values on the cyclic modules of the form $\Z[t]/\mathfrak p$, where $\mathfrak p$ is a prime ideal of $\Z[t]$.
The ``Bernoulli normalization axiom'' gives the value of the algebraic entropy in the first case, while the ``Yuzvinski axiom'' in the second case.  A different relation between algebraic entropy and length functions can be found in \cite{SVV}.

\smallskip
In the final Section \ref{mahler-sec} we consider the relation of the algebraic entropy to the Mahler measure and Lehmer Problem. We give the proof of Theorem \ref{epsilon}, making use of the Addition Theorem. Moreover, we end the section and the paper by proving the Realization Theorem.

\bigskip
The deduction of the Addition Theorem from the Algebraic Yuzvinski Formula was obtained already in \cite{DG-peters}. But at that time no direct proof of the Algebraic  Yuzvinski Formula was available. On the other hand, its validity for automorphisms of finite dimensional rational vector spaces could be deduced from its counterpart for the topological entropy and Theorem \ref{Peters}(b).  As it was underlined in \cite{DG-peters}, ``the value of the Addition Theorem would be much higher if a purely algebraic proof of the Algebraic Yuzvinski Formula would be available''. 
Indeed, we apply the Algebraic Yuzvinski Formula in the proof of the Addition Theorem and from these two results we deduce the Uniqueness Theorem and the Realization Theorem. As mentioned above, another important application is given in \cite{DG2}, where Theorem \ref{Weiss}(a) is extended to a general Bridge Theorem for all the endomorphisms of arbitrary abelian groups, obtaining in this way also a completely new proof of Theorem \ref{Peters}(b) (the original proof makes a heavy use of convolutions). Finally, using the Bridge Theorem and the Uniqueness Theorem for the algebraic entropy, one can obtain also a Uniqueness Theorem for the topological entropy of the endomorphisms of compact abelian groups (see \cite[Corollary 3.3]{DG2}).

\subsection*{Notation and terminology}

We denote by $\mathbb Z$, $\mathbb N$, $\mathbb N_+$, $\Q$ and $\R$ respectively the set of integers, the set of natural numbers, the set of positive integers, the set of rationals and the set of reals. Let $\R_{\geq0}=\{r\in\R:r\geq0\}$. For $m\in\mathbb N_+$, we use $\mathbb Z(m)$ for the finite cyclic group of order $m$. 

Let $G$ be an abelian group. With a slight divergence with the standard use, we denote by $[G]^{<\omega}$ the set of all \emph{non-empty} finite subsets of $G$. If $H$ is a subgroup of $G$, we indicate this by $H\leq G$. For a subset $X$ of $G$ we denote by $\hull{X}$ the subgroup of $G$ generated by $X$. 
The free rank of $G$ is denoted by $r(G)$. 
The subgroup of torsion elements of $G$ is $t(G)$, while $D(G)$ denotes the divisible hull of $G$. For a cardinal $\alpha$ we denote by $G^{(\alpha)}$ the direct sum $\bigoplus_\alpha G$ of $\alpha$ many copies of $G$. 

Moreover, $\End(G)$ is the ring of all endomorphisms of $G$. We denote by $0_G$ and $\id_G$ respectively the endomorphism of $G$ which is identically $0$ and the identity endomorphism of $G$. If $G$ is torsion-free and $\phi\in \End(G)$, we denote by $\widetilde\phi:D(G)\to D(G)$ the unique extension of $\phi$ to $D(G)$.

If $G$ is an abelian group, $\phi\in\End(G)$ and $H$ is a $\phi$-invariant subgroup of $G$, we denote by $\overline\phi_H\in\End(G/H)$ the endomorphism induced by $\phi$ on the quotient $G/H$. When there is no possibility of confusion, we write simply $\overline\phi$.

If $R$ is a ring, $M$ is an $R$-module and $r\in R$, we denote by $\mu_r:M\to M$ the group endomorphism of $M$ defined by $\mu_r(x)=rx$ for every $x\in M$.
To avoid confusion, we write also $\mu_r^M$, when more precision is necessary.

\section{Basic properties of entropy}\label{sec-def}

Let $G$ be an abelian group and $\phi\in \End(G)$. For $F\in[G]^{<\omega}$, the $\phi$-trajectory $T(\phi,F)$ need not be a subgroup of $G$, so let $$V(\f,F)=\hull{\f^n(F):n\in\N}=\hull{T(\phi,F)}.$$ This is the smallest $\f$-invariant subgroup of $G$ containing $F$. 
For $g\in G$, we write simply $V(\phi,g)$ in place of $V(\phi,\{g\})$. So, for $F\in[G]^{<\omega}$, one has $$V(\f,F)=\sum_{g\in F}V(\phi,g).$$

\begin{definition}\label{fg-def}
Let $G$ be an abelian group and $\phi\in \End(G)$. We say that $G$ is \emph{finitely $\phi$-generated} by $F\in[G]^{<\omega}$ if $G=V(\phi,F)$.
\end{definition}

As a start, we prove that the limit in \eqref{lim-eq} defining the algebraic entropy exists. 

\begin{lemma}\label{subadd}\label{lim-ex}
Let $G$ be an abelian group, $\phi\in\End(G)$ and $F\in[G]^{<\omega}$. For every $n\in\N_+$ let $c_n=\log|T_n(\phi,F)|$.
\begin{itemize}
\item[(a)] The sequence $\{c_n\}_{n\in\N_+}$ is subadditive, i.e., $c_{n+m}\leq c_n+c_m$ for every $n,m\in\N_+$.
\item[(b)] The limit $H(\phi,F)=\lim_{n\to\infty}\frac{c_n}{n}$ exists and $H(\phi,F)=\inf_{n\in\N_+}\frac{c_n}{n}$.
\end{itemize}
\end{lemma}
\begin{proof}
(a) By definition 
\begin{align*}
T_{n+m}(\phi,F) &=F+\phi(F)+\ldots+\phi^{n-1}(F)+\phi^n(F)+\ldots+\phi^{n+m-1}(F)\\
&=T_n(\phi,F)+\phi^n(T_m(\phi,F)).
\end{align*}
Consequently,
\begin{equation*}
c_{n+m} =\log|T_{n+m}(\phi,F)|\leq \log(|T_n(\phi,F)|\cdot |T_m(\phi,F)|)=c_n+c_m.
\end{equation*}

\smallskip
(b) By item (a) the sequence $\{c_n\}_{n\in\N_+}$ is subadditive. Then the sequence $\{\frac{c_n}{n}:n\in\N_+\}$ has limit and $\lim_{n\to\infty}\frac{c_n}{n}=\inf_{n\in\N_+}\frac{c_n}{n}$ by a folklore fact from Calculus, due to Fekete (see \cite{Fek}).
\end{proof}

\begin{remark}\label{0inF} Let $G$ be an abelian group and $\phi\in\End(G)$.
\begin{itemize}
\item[(a)] The function $H(\phi,-)$ is non-decreasing, that is, $H(\phi,F)\leq H(\phi,F')$ for every $F,F'\in[G]^{<\omega}$ with $F\subseteq F'$.
\item[(b)] By  item (a), if $\mathcal F$ is a cofinal subfamily of $[G]^{<\omega}$ with respect to inclusion, then $h(\phi)=\sup\{H(\phi,F):F\in\mathcal F\}$.
In particular, it is possible to calculate $h(\phi)$ supposing without loss of generality that $0\in F$, that is, $$h(\phi)=\sup\{H(\phi,F):F\in[G]^{<\omega}, 0\in F\}.$$
\end{itemize}
\end{remark}

The definition of $H(\phi,F)$ in \eqref{lim-eq} shows that the algebraic entropy $H(\phi,F)$ of $\f$ with respect to $F\in[G]^{<\omega}$ measures the growth rate of the partial trajectories $T_n(\phi,F)$  when they approximate the trajectory $T(\phi,F)$. Our next comment concerns the case when $T(\phi,F)= V(\phi,F)$ is a subgroup. This occurs for example when $F$ is a finite {\em subgroup} of $G$. In such a case, it was proved in \cite{DGSZ} that $H(\phi,F) = \log m$, where $m > 0 $ is a divisor of $|F|$.  The number $m$ can be found, without any computation of limits, as follows. 

\begin{fact}\label{Fact:limit:free}
\emph{Let $G$ be an abelian group, $\phi\in\End(G)$ and $F$ is finite subgroup of $G$.
First, we note that $N = T(\phi,F)$ is a torsion subgroup of $G$. Moreover, with $\f_1 = \f\restriction_{N}$, one has  $H(\phi,F)= H(\phi_1,F)$, i.e., the computation of $H(\phi,F)$ can be carried out in the torsion $\f$-invariant subgroup $N$ of $G$. 
The equality $N= F +\phi(N)$ shows that $N/\phi(N)$ is a quotient of the finite subgroup $F$, so $ |N/\phi(N)|$ is a divisor of $|F|$. It was proved in \cite{DG-lf} that $H(\phi,F) = \log  |N/\phi(N)|$, provided $\f$ is injective. This formula was stated without a proof and without asking $\f$ to be injective by Yuzvinski in \cite{Y}. A first proof was provided in \cite{DSV}, where it was pointed out that the formula is wrong if $\f$ is not injective (e.g., when $\f = 0_G$ and $|F|> 1$). Finally, in \cite[Lemma 4.1]{DG-lf} $\ker\phi\cap N$ was proved to be finite and the following more precise equality was obtained in the general case: 
$$
H(\phi,F) = \log \left |\frac{T(\phi,F)}{\phi(T(\phi,F))} \right|-\log|\ker\phi\cap T(\phi,F)|.
$$ }
\end{fact}

We consider now the two most natural examples. 

\begin{example}\label{h(id)=0}
If $G$ is an abelian group, then $h(0_G)=h(\id_G)=0$.

The equality $h(0_G)=0$ is obvious. 
To show that also $h(\id_G)=0$, let $F=\{f_1,\ldots,f_t\}$ be a finite subset of $G$. As $T_n(\id_G,F)= \underbrace{F+\ldots +F}_n$ for all $n\in\N_+$,
every $x\in T_n(\id_G,F)$ can be written as $x=\sum_{i=1}^t m_i f_i$, for some $m_i\in\N$ with $\sum_{i=1}^t m_i=n$.  Clearly, $(m_1,\ldots,m_t)\in\{0,1,\ldots,n\}^t$,  so $|T_n(\id_G,F)|\leq (n+1)^t$. Hence, $H(\id_G,F)=0$, and consequently $h(\id_G)=0$ by the arbitrariness of $F\in[G]^{<\omega}$. 
\end{example}

For $G$ an abelian group and $\phi\in\End(G)$, the \emph{hyperkernel} of $\phi$ is $$\ker_\infty \f = \bigcup_{n\in\N_+} \ker \f^n.$$ The subgroup $\ker_\infty \f$ is $\f$-invariant and also invariant for inverse images. Hence, the induced endomorphism $\overline \f: G /\ker_\infty \f \to G /\ker_\infty \f$ is injective.

\begin{proposition}\label{ker_infty}
Let $G$ be an abelian group and $\phi\in\End(G)$. Then $h(\f\restriction_{\ker_\infty\f})=0$.
\end{proposition}
\begin{proof}
Let $F\in[G]^{<\omega}$ with $F\subseteq \ker_{\infty}\phi$ and assume without loss of generality that $0\in F$ (see Remark \ref{0inF}(b)). There exists $m\in\N_+$ such that $\phi^m(F)=0$. Consequently, $T_n(\phi,F)=T_m(\phi,F)$ for every $n\in\N$ with $n\geq m$. Hence, $H(\phi,F)=0$, 
and we conclude that $h(\phi\restriction_{\ker_\infty\phi})=0$ by the arbitrariness of $F$.
\end{proof}

In the next lemma we show that $h$ is monotone under taking restrictions to invariant subgroups and under taking induced endomorphisms on quotients over invariant subgroups.

\begin{lemma}\label{restriction_quotient}
Let $G$ be an abelian group, $\phi\in\End(G)$ and $H$ a $\phi$-invariant subgroup of $G$. Then $h(\f)\geq \max\{h(\f\restriction_H),h(\overline{\f})\}.$
\end{lemma}
\begin{proof}
For every $F\in[H]^{<\omega}$, obviously $H(\phi\restriction_H,F)=H(\phi,F)$, so $H(\phi\restriction_H, F)\leq h(\phi)$. Hence, $h(\phi\restriction_H)\leq h({\phi})$.

Now assume that $F\in[G/H]^{<\omega}$ and $F=\pi(F_0)$ for some $F_0\in[G]^{<\omega}$, where $\pi:G\to G/H$ is the canonical projection. Then $\pi(T_n(\phi,F_0))= T_n(\overline{\phi},F)$ for every $n\in\N_+$. Therefore, $ h(\phi)\geq H(\phi,F_0)\geq H(\overline{\phi}, F)$ and by the arbitrariness of $F$ this proves $h(\phi)\geq h(\overline{\phi})$.
\end{proof}

The next proposition collects the basic properties of $h$, which are also typical properties of the known entropy functions. Indeed, they are inspired by similar properties of the algebraic entropy $\ent$ (see \cite{W} and \cite{DGSZ}) and the topological entropy (see \cite{AKM} and \cite{St}). In the case of $h$ they were proved in \cite{Pet} for automorphisms, here we extend them for endomorphisms. We shall refer to them as, respectively, \emph{Invariance under conjugation}, \emph{Logarithmic Law}, \emph{Continuity for direct limits} and \emph{weak Addition Theorem}.

\begin{proposition}\label{properties}
Let $G$ be an abelian group and $\phi\in\End(G)$.
\begin{itemize}
\item[(a)] If $H$ is another abelian group, $\eta\in\End(H)$ and $\phi$ and $\eta$ are conjugated, then $H(\phi,F)=H(\eta,\xi(F))$ for every $F\in[G]^{<\omega}$; in particular, $h(\phi)=h(\eta)$.
\item[(b)] For every $k\in\N_+$, $h(\phi^k) = k h(\phi)$. If $\phi$ is an automorphism, then $h(\phi^k) = |k|h(\phi)$ for every $k\in\Z$.
\item[(c)] If $G$ is a direct limit of $\phi$-invariant subgroups $\{G_i:i\in I\}$, then $h(\phi)=\sup_{i\in I}h(\phi\restriction_{G_i})$.
\item[(d)] If $G=G_1\times G_2$, $\phi=\phi_1\times \phi_2$ with $\phi_i\in\End(G_i)$ and $F_i\in[G_i]^{<\omega}$, $i=1,2$, then $H(\phi,F_1\times F_2)=H(\phi_1,F_1)+H(\phi_2,F_2)$; consequently, $h(\phi_1\times\phi_2)=h(\phi_1)+h(\phi_2)$
\end{itemize}
\end{proposition}
\begin{proof}
(a) For $F\in[G]^{<\omega}$ and $n\in\N_+$, $T_n(\eta,\xi(F))=\xi(F)+\xi(\phi(F))+\ldots+\xi(\phi^{n-1}(F))$, as $\xi\phi^k = \phi^k \xi$ for $k= 1,2, \ldots ,n-1$. Since $\xi$ is an isomorphism, $|T_n(\f, F)|=|T_n(\eta, \xi(F))|$, and so $H(\phi,F)=H(\eta,\xi(F))$. This proves that $h(\phi)=h(\eta)$.

\smallskip
(b) Fix $k\in\N_+$. First, we prove the inequality $h(\phi^k)\leq k h(\phi)$.
Let $F\in[G]^{<\omega}$, assuming without loss of generality that $0\in F$ (see Remark \ref{0inF}(b)). Let $n\in \N_+$, then $T_n(\phi^k,F)\subseteq T_{kn-k +1}(\phi,F)$ and so
\begin{align*}
H(\phi^k,F)&=\lim_{n\to\infty}\frac{\log|T_n(\phi^k,F)|}{n}  \leq \lim_{n\to\infty}\frac{\log|T_{kn-k+1}(\phi,F)|}{n}\\
& =\lim_{n\to\infty} \frac{\log|T_{kn-k+1}(\phi,F)|}{kn-k+1}\cdot \lim_{n\to\infty}\frac{kn-k+1}{n}\\
& =k \lim_{n\to\infty} \frac{\log|T_{kn-n+1}(\phi,F)|}{kn-k+1}=kH(\phi,F)\leq k h(\phi).
\end{align*}
Therefore, $h(\phi^k)\leq k h(\phi)$.

To check the inequality $h(\phi^k)\geq k h(\phi)$, let $F\in[G]^{<\omega}$ and $n\in \N_+$. With $F_1=T_k(\phi,F)$ one has 
$T_n(\phi^k,F_1)=T_{kn}(\phi,F)$. Then 
$$\frac{1}{k}h(\phi^k) \geq  \frac{1}{k}H(\phi^k,F_1)=\lim_{n\to\infty}\frac{\log|T_n(\phi^k,F_1)|}{kn} =\lim_{n\to\infty}\frac{\log|T_{kn}(\phi,F)|}{kn}=H(\phi,F).$$ 
Therefore, $h(\phi^k)\geq k h(\phi)$. 

Now assume that $\phi$ is an automorphism. It suffices to prove that $h(\phi^{-1})=h(\phi)$. Let $F\in[G]^{<\omega}$ and $n\in \N_+$. Then $T_n(\phi^{-1},F)=\phi^{-n+1}(T_n(\phi,F))$; in particular, $|T_n(\phi^{-1}, F)|=|T_n(\phi,F)|$, as $\phi$ is an automorphism. This yields $H(\phi^{-1},F)=H(\phi,F)$, hence $h(\phi^{-1})=h(\phi)$.

\smallskip
(c) By Lemma \ref{restriction_quotient}, $h(\f)\geq h(\f\restriction_{G_i})$ for every $i\in I$ and so $h(\f)\geq\sup_{i\in I}h(\f\restriction_{G_i})$.
To check the converse inequality, let $F\in[G]^{<\omega}$. Since $G=\varinjlim\{G_i:i\in I\}$ and $\{G_i:i\in I\}$ is a directed family, there exists $j\in I$ such that $F\subseteq G_j$. Then $H(\f,F)=H(\f\restriction_{G_j},F)\leq h(\f\restriction_{G_j})$. This proves that $h(\f)\leq\sup_{i\in I}h(\f\restriction_{G_i})$.

\smallskip
(d) For every $n\in\N_+$, $$T_n(\phi,F_1\times F_2)=T_n(\phi_1,F_1)\times T_n(\phi_2, F_2)$$ 
Hence, 
\begin{equation}\label{times-eq}
H(\phi,F_1\times F_2)=H(\phi_1,F_1)+H(\phi_2,F_2).
\end{equation} 
Consequently, $h(\phi)\geq h(\phi_1)+h(\phi_2)$. Since every $F\in[G]^{<\omega}$ is contained in some $F_1\times F_2$, for $F_i\in[G_i]^{<\omega}$, $i=1,2$, and so $H(\phi,F)\leq H(\phi,F_1\times F_2)$. Therefore, \eqref{times-eq} proves also that $h(\phi)\leq h(\phi_1)+h(\phi_2)$.
\end{proof}

The following easy remark can be deduced directly either from item (a) or from item (b) of Proposition \ref{properties}: if $G$ is an abelian group and $\phi\in\End(G)$, then 
\begin{equation}\label{-f}
h(-\phi)=h(\phi).
\end{equation}

The next is a direct consequence of Proposition \ref{properties}(b).

\begin{corollary}\label{0<->0}
Let $G$ be an abelian group and $\f\in\End(G)$. Then:
\begin{itemize}
\item[(a)]$h(\f)=0$ if and only if $h(\f^k)=0$ for some $k\in\N_+$;
\item[(b)]$h(\f)=\infty$ if and only if $h(\f^k)=\infty$ for some $k\in\N_+$.
\end{itemize}
\end{corollary}

The next somewhat technical consequence of Proposition \ref{properties}(a) is frequently applied in the following sections. Let $G$ be an abelian group, $\f\in\End(G)$ and $H,K$ two $\f$-invariant subgroups of $G$. Consider $\overline\phi_H\in\End(G/H)$ and let $\pi: G\to G/H$ be the canonical projection. The subgroup $\pi(K)$ of $G/H$ is $\overline\f_H$-invariant, and $\pi(K)=K+H/H\cong K/H\cap K$. Let $\overline{\f \restriction_{K}}\in\End(K/H\cap K)$ be the induced endomorphism, then $\overline\f _H\restriction_{\pi(K)}$ is conjugated to $\overline{\f \restriction_{K}}$, so Proposition \ref{properties}(a) gives the equality
\begin{equation}\label{restr-to-pi}
h(\overline{\f}_H\restriction_{\pi(K)})=h(\overline{\f \restriction_{K}}).
\end{equation}

In the following example we compute the algebraic entropy of the Bernoulli shift.

\begin{example}\label{beta} 
For any abelian group $K$, $$h(\beta_K)=\log|K|,$$ with the usual convention that $\log|K|=\infty$, if $|K|$ is infinite.

It is proved in \cite{DGSZ} that $h(\beta_K)=\ent(\beta_K)=\log|K|$ for every finite abelian group $K$. We verify now that $h(\beta_\Z)=\infty$. Indeed, let $G=\Z^{(\N)}$; for every prime $p$, the subgroup $pG$ of $G$ is $\beta_\Z$-invariant, so $\beta_\Z$ induces an endomorphism $\overline{\beta_\Z}:G/pG\to G/pG$. Since $G/pG\cong \Z(p)^{(\N)}$, and $\overline{\beta_\Z}$ is conjugated to $\beta_{\Z(p)}$ through this isomorphism, $h(\overline{\beta_\Z})=h(\beta_{\Z(p)})=\log p$ by Proposition \ref{properties}(a). Therefore, $h(\beta_\Z)\geq\log p$ for every prime $p$ by Lemma \ref{restriction_quotient}, and so $h(\beta_\Z)=\infty$. 

Assume that $K$ is an infinite abelian group. If $K$ is non-torsion, then $K$ contains a subgroup $C\cong \Z$, so $K^{(\N)}$ contains the $\beta_K$-invariant subgroup $C^{(\N)}$ isomorphic to $\Z^{(\N)}$. Moreover, $\beta_K\restriction_{C^{(\N)}} = \beta_C$ is conjugated to $\beta_\Z$. Hence, by Lemma \ref{restriction_quotient}, Proposition \ref{properties}(a) and the previous part of this example, $h(\beta_K)\geq h(\beta_C)=h(\beta_\Z)=\infty$. If $K$ is torsion, then $K$ contains arbitrarily large finite subgroups $H$. Consequently, $K^{(\N)}$ contains the $\beta_K$-invariant subgroup $H^{(\N)}$. By Lemma \ref{restriction_quotient} and the first part of this example, $h(\beta_K)\geq h(\beta_H)=\log|H|$ for every $H$. So $h(\beta_K)=\infty$.
\end{example}

In \cite{AADGH}, extending the notion of Bernoulli shift, for every self-map $\lambda: X \to X$ of a non-empty set $X$, the generalized shift $\sigma_\lambda$ of an arbitrary Cartesian power $K^X$ of an abelian group $K$ was introduced by letting $\sigma_\lambda((x_i)_{i\in X})=(x_{\lambda(i)})_{i\in X}$ for every $(x_i)_{i\in X}\in K^X$. In case $\lambda$ is finitely many-to-one, the subgroup $K^{(X)}$ is $\sigma_\lambda$-invariant, so one can consider also the generalized shift $\sigma_\lambda: K^{(X)} \to K^{(X)}$.  In \cite{AADGH} and \cite{GB}, $h(\sigma_\lambda)$ was computed on direct sums and on direct products, respectively.

\begin{remark}
There is a remarkable connection between entropy and recurrence in the spirit of Poincar\' e recurrence theorem.    
We do not pursue this topic in the present paper, here we only recall some recent results from  \cite{DG}. 

Let $G$ be an abelian group and $\phi\in\End(G)$. An element $x\in G$ is \emph{quasi-periodic} if there exist $n>m$ in $\N$ such that $\phi^n(x)=\phi^m(x)$; $\f$ is \emph{locally quasi-periodic} if every $x\in G$ is a quasi-periodic point of $\phi$. Moreover, $\phi$ is \emph{quasi-periodic} if there exist $n>m$ in $\N$ such that $\phi^n(x)=\phi^m(x)$ for every $x\in G$. If $G$ is finitely $\phi$-generated and $\f$ is locally quasi-periodic, then $\f$ is quasi-periodic. The set $Q(G,\f)$ of all quasi-periodic points of $\phi$ in $G$ is a $\f$-invariant subgroup of $G$ containing $\ker_\infty \f$. 

 It follows from Proposition \ref{properties}(b) that either $h(\f) = 0$ or $h(\f) = \infty$ whenever $\f$ is quasi-periodic, but one can prove more.
According to \cite{DG}, there exists a greatest $\phi$-invariant subgroup $P=P(G,\f)$ of $G$ such that $h(\phi\restriction_P)=0$. Moreover, $P(G/P,\bar \f) =0$ and $P$ contains $Q(G,\f)$ (so $h(\f) = 0$ in case $\f$ is {\em locally} quasi-periodic). On the other hand, $\f\in \End(G)$ is called \emph{algebraically ergodic} if $Q(G,\f)= 0$. Then $\phi$ is algebraically ergodic if and only if $P$ is trivial. Therefore, $\overline \f \in \End(G/P)$ is algebraically ergodic and $h(\f)=h(\overline\phi)$ (see \cite{DG}).
\end{remark}

\section{Entropy on torsion-free abelian groups}\label{tf-sec}

Here we dedicate more attention to the torsion-free abelian groups. We start by computing the algebraic entropy of the endomorphisms of $\Z$, 
for the counterpart for $\Q$ see Example \ref{Qex}. 

\begin{example}\label{Zex}
\begin{itemize}
\item[(a)] For every non-zero $k\in\Z$, $\mu_k\in\End(\Z)$ has $h(\mu_k)=\log |k|$. As $h(\mu_k) =h(\mu_{-k})$ by \eqref{-f}, we may assume without loss of generality that $k\in \N_+$.

For $k=1$ the conclusion follows from Example \ref{h(id)=0}. Assume that $k>1$ and let $F_0=\{0,1,\ldots,k-1\}\in[\Z]^{<\omega}$ and $n\in\N_+$. Every $m\in\N$ with $m<k^n$ can be uniquely written in the form $m=f_0+f_1k+\ldots+f_{n-1}k^{n-1}$ with all $f_i\in F_0$. Then $T_n(\mu_k,F_0)=\{m\in\N:m<k^n\}$, and so $|T_n(\mu_k,F_0)|=k^n$. Consequently $H(\mu_k,F_0)=\log k$ and this yields $h(\mu_k)\geq \log k$. 
To prove the converse inequality $h(\mu_k)\leq\log k$, fix $m\in\N_+$ and let $F_m=\{0,\pm1,\pm 2,\ldots,\pm m\}\in[\Z]^{<\omega}$ and $n\in\N_+$. Then $|x|\leq mk^n$ for every $x\in T_n(\mu_k,F_m)$. So $|T_n(\mu_k,F_m)|\leq 3mk^n$, hence $H(\mu_k,F_m)\leq \log k$. Since each $F\in[\Z]^{<\omega}$ is contained in $F_m$ for some $m\in\N_+$, we obtain $h(\mu_k)\leq \log k$.

\item[(b)] For $n,k\in\N_+$ and $\mu_k:\Z^n\to\Z^n$, $h(\mu_k)=n\log k$. This follows from Proposition \ref{properties}(d) and item (a).
\end{itemize}
\end{example}

Item (a) of the above example shows a difference with the torsion case. In fact, for an abelian group $G$ and $\phi\in\End(G)$, $\phi$ is \emph{integral} if there exists $f(t)\in\Z[t]\setminus\{0\}$ such that $f(\phi)=0$. According to \cite[Lemma 2.2]{DGSZ}, if $G$ is torsion and $\phi$ is integral, then $\ent(\phi)=0$. On the other hand, for $k>1$ the endomorphism $\mu_k:\Z\to \Z$ in (a) of the above example is integral over $\Z$ (as $\mu_k(x)-kx=0$ for all $x\in\Z$), nevertheless, $h(\mu_k)=\log k>0$.

\begin{example} 
Fix $k\in \Z$ and consider the automorphism $\f:\Z^2\to \Z^2$ defined by $\f(x,y)= (x+ky, y)$ for all $(x,y)\in\Z^2$. Then $h(\f) =0$. 

Let $m\in\N_+$ and $F_m=\{0,\pm1,\pm 2,\ldots,\pm m\}\times \{0,\pm1,\pm 2,\ldots,\pm m\}\in[\Z^2]^{<\omega}$. Every $F\in[\Z^2]^{<\omega}$ is contained in $F_m$ for some $m\in\N$. Therefore, it suffices to show that $H(\f,F_m)=0$. One can prove by induction that, for every $n\in\N_+$, $T_n(\f,F_m)$ is contained in a parallelogram with sides $2nm$ and $nm(2+nk-k)$, so $|T_n(\f,F_m)|\leq 2n^2m^2(2+nk-k)$. Thus $H(\f,F_m)=0$ for every $m\in\N$ and so $h(\phi)=0$.
\end{example}

Next we give some properties of the algebraic entropy for endomorphisms $\phi$ of (finitely $\phi$-generated) abelian groups.

\begin{lemma}\label{X}
Let $G$ be an abelian group and $\f\in\End(G)$. If  $V(\phi,g)$ has infinite rank for some $g\in G$, then  
$V(\phi,g)$ is torsion-free and $h(\phi\restriction_{V(\phi,g)})=h(\phi)=h(\overline\phi)=\infty$, where $\overline\phi\in \End(G/t(G))$ is the induced endomorphism.
\end{lemma}

\begin{proof} Let $H:=V(\phi,g)$. To prove that $H$ is  torsion-free consider the canonical projection  $\pi:H\to H/t(H)$, where $H/t(H)= \pi(V(\f,g))=V(\overline\f,\pi(g))$ is torsion-free and has infinite rank.  Hence, $\hull{\overline\f^{m}(\pi(g))}\cap T_m(\overline\f,\hull{\pi(g)})=0$ for every $m\in\N_+$, so $T_m(\overline\f,\hull{\pi(g)})=\bigoplus_{k=0}^{m-1}\hull{\overline\f^k(\pi(g))}$ and $\pi(V(\f,g))=\bigoplus_{k\in\N}\hull{\overline\f^k(\pi(g))}\cong\Z^{(\N)}$. 
Suppose that $x\in H$ is torsion, then $\pi(x)$ is torsion as well, so $\pi(x) = 0$ as $H/t(H)$ is torsion-free. If $x = \sum_{i=0}^mk_i \f^i(g)$, then $\pi(x) = \sum_{i=0}^mk_i \overline\f^i(\pi(g)) = 0$. This entails $k_0=\ldots =k_m=0$, so $x=0$. Hence, $t(H) = 0$, and so $H$ is torsion-free. 

Since $\f\restriction_{H}$ is conjugated to $\beta_\Z$, by Proposition \ref{properties}(a) and Example \ref{beta} we conclude that $h(\f\restriction_H)=h(\beta_\Z)=\infty$. Therefore, $h(\overline\phi)=\infty$ and $h(\f)=\infty$ by Lemma \ref{restriction_quotient}. 
\end{proof}

The following are consequences of Lemma \ref{X}.

\begin{corollary}\label{cyc_case}
Let $G$ be a torsion-free abelian group and $\f\in\End(G)$. 
\begin{itemize}
\item[(a)] If $G$ is finitely $\f$-generated, then $G$ has finite rank if and only if $h(\f)<\infty$. 
\item[(b)] If $h(\phi\restriction_{V(\phi,g)})=\infty$ for some $g\in G$, then $h(\f\restriction_{V(\phi,z)})=\infty$ for every $z\in V(\phi,g)\setminus\{0\}$.
\end{itemize}
\end{corollary}
\begin{proof}
(a) If $r(G) = n$ is finite, then $D(G)\cong \Q^n$. Theorem \ref{Yuz} gives $h(\widetilde\phi)<\infty$, and so $h(\f)\leq h(\widetilde\phi)<\infty$ by Lemma \ref{restriction_quotient}. 

Now assume that $h(\phi)<\infty$. Since $G$ is finitely $\f$-generated, then $G$ is a sum of $\phi$-invariant subgroups of the form $V(\phi,g_i)$ for finitely many $g_i\in G$. Lemma \ref{X} implies that all $r(V(\phi,g_i))$ are finite, hence $r(G)$ is finite as well.

\smallskip
(b) By item (a), $h(\phi\restriction_{V(\phi,g)})=\infty$ implies $r(V(\phi,g))$ infinite. For $z\in V(\phi,g)\setminus\{0\}$, it is easy to see that $r(V(\phi,z))$ is infinite as well, and so $h(\f\restriction_{V(\phi,z)})=\infty$ by Lemma \ref{X}.
\end{proof}

By Corollary \ref{cyc_case}(a), if $G$ is a torsion-free abelian group, $\phi\in\End(G)$ and $g \in G$, then
\begin{center}
$h(\f\restriction_{V(\phi,g)})=\infty$ if and only if $r(V(\phi,g))$ is infinite.
\end{center}
Moreover, in view of Corollary \ref{cyc_case}(b), if $h(\phi\restriction_{V(\phi,g)})=\infty$, then for every $\f$-invariant subgroup $H$ of $V(\phi,g)$ one has the following striking dichotomy:
\begin{center}
either $H=0$ or $h(\f\restriction_H)=\infty$.  
\end{center}

The next is a consequence of Lemma \ref{X} and Corollary \ref{cyc_case} that completes Lemma \ref{X} from the point of view of the Addition Theorem.

\begin{lemma}\label{CorollaryNov8:1} Let $G$ be an abelian group and $\f\in\End(G)$ such that $V(\f,g)$ has infinite rank for some $g\in G$. Then $h(\f)=\infty$ and $\AT_h(G,\f,H)$ for every subgroup $H$ of $G$.
 \end{lemma}

\begin{proof} As $V(\f,g)$ has infinite rank,  $h(\f)=\infty$ and $V(\f,g)$ is torsion-free by Lemma \ref{X}. Let $H$ be a subgroup of $G$ and suppose that $V(\f,g)\cap H \ne  0$. As $V(\f,g)$ is torsion-free, we can apply Corollary \ref{cyc_case}(b) to $\f\restriction_{V(\f,g)}$ to deduce that $h(\f\restriction_{V(\f,g)\cap H})=\infty$. By Lemma \ref{restriction_quotient}, $h(\f\restriction_H)=\infty$ as well. Then $h(\f)=h(\f\restriction_H)=\infty$ and in particular $\AT_h(G,\f,H)$. 
 
Assume now that $V(\f,g)\cap H=0$. Let $\pi:G\to G/H$ be the canonical projection. Therefore, $V(\f,g)$ projects injectively in the quotient $G/H$ and so $\pi(V(\f,g))=V(\overline\f,\pi(g))$ has infinite rank. So $h(\overline \f)=\infty$ by Lemma \ref{X}. Hence, $h(\f)=h(\overline\f)=\infty$ by Lemma \ref{restriction_quotient} and in particular $\AT_h(G,\f,H)$. 
\end{proof}

A subgroup $H$ of an abelian group $G$ is \emph{essential} if and only if for every $x\in G\setminus\{0\}$ there exists $k \in \Z$ such that $kx \in H\setminus\{0\}$. In case $G$ is torsion-free, $H$ is an essential subgroup of $G$ if and only if $G/H$ is torsion.

\begin{corollary}\label{CoAAA} 
Let $G$ be a torsion-free abelian group, $\f\in\End(G)$ and let $H$ be an essential $\f$-invariant subgroup of $G$. If $h(\phi)<\infty$, then $h(\overline\phi)=0$. 
\end{corollary}
\begin{proof} Since $G/H$ is torsion, it suffices to see that $h(\overline\f\restriction_{(G/H)[p]})=0$ for every prime $p$, by \cite[Proposition 1.18]{DGSZ}. To this end we have to show that every $\overline x\in (G/H)[p]$ has finite trajectory under $\overline \f$. By Lemma \ref{X}, $V(\phi,x)$ has finite rank, say $n\in\N$. Then there exist $k_i\in\Z$, $i=0,\ldots,n$, such that $\sum_{i=0}^{n}k_i\f^i(x)=0$. Since $G$ is torsion-free, we can assume without loss of generality that at least one of these coefficient is not divisible by $p$. Now, projecting in $G/H$, we conclude that $\sum_{i=0}^{n}k_i\overline\f^i(\overline x)=0$ is a non-trivial linear combination in $(G/H)[p]$. Hence, $\overline x\in (G/H)[p]$ has finite trajectory under $\overline \f$.
\end{proof}

The next result reduces the computation of the algebraic entropy for endomorphisms of torsion-free abelian groups to the case of endomorphisms of \emph{divisible} torsion-free abelian groups. Item (b) was announced without proof by Yuzvinski in \cite{juz67}.
   
\begin{proposition}\label{AA}
Let $G$ be a torsion-free abelian group and $\f\in\End(G)$. Then:
\begin{itemize}
\item[(a)] if $F\in[D(G)]^{<\omega}$ and $m\in\N_+$ satisfy $mF\subseteq G$, then $H(\phi,mF)=H(\widetilde\phi,mF)=H(\widetilde\phi,F)$;
\item[(b)]  $h(\f)=h(\widetilde\f)$.
\end{itemize}
\end{proposition}
\begin{proof}
(a) The automorphism $\mu_m$ of $D(G)$ commutes with $\widetilde \f$. Hence, Proposition \ref{properties}(a) yields $H(\f,mF)=H(\widetilde\phi,mF)=H(\widetilde \f,F)$. 

\smallskip
(b) Obviously, $h(\f)\leq h(\widetilde\f)$ by Lemma \ref{restriction_quotient}.  By the arbitrariness of $F$ in item (a), we have also $h(\f)\geq h(\widetilde\f)$.
\end{proof}

Proposition \ref{AA} may fail if the abelian group $G$ is not torsion-free. Indeed, for $G = \Z(2)^{(\N)}$ and $\beta_{\Z(2)} \in \End(G)$ one has $h(\beta_{\Z(2)})=\ent(\beta_{\Z(2)})=\log 2$. On the other hand, for $D(G)=\Z(2^\infty)^{(\N)}$ the endomorphism $\beta_{\Z(2^\infty)}$ extends $\beta_{\Z(2)}$ and has 
$h(\beta_{\Z(2^\infty)})=\ent(\beta_{\Z(2^\infty)})=\infty$ (see Example \ref{beta}).

\medskip
The properties from the next lemma, frequently used in the sequel, are well known and easy to prove.
Recall that a subgroup $H$ of an abelian group $G$ is \emph{pure} in $G$ if $H\cap mG=mH$ for every $m\in\Z$. The \emph{purification} $H_*$ of $H$ in $G$ is the smallest pure subgroup of $G$ containing $H$.

\begin{lemma}\label{div<->pure}
Let $G$ be an abelian group and $H$ a subgroup of $G$.
\begin{itemize}
\item[(a)] If $H$ is divisible, then $H$ is pure.
\item[(b)] If $G$ is divisible, then $H$ is divisible if and only if $H$ is pure.
\item[(c)] If $G$ is torsion-free, then $H$ is pure if and only if $G/H$ is torsion-free.
\end{itemize}
\end{lemma}

Now we see that in the torsion-free case it suffices to compute the algebraic entropy in an essential invariant subgroup. As a by-product, this proves $\AT_h(G,\phi,H)$ for a torsion-free abelian group $G$, $\phi\in\End(G)$ and an essential $\phi$-invariant subgroup $H$ of $G$.

\begin{corollary}\label{CoAA}\label{AT-tf-ess}
Let $G$ be a torsion-free abelian group, $\f\in\End(G)$ and $H$ a $\f$-invariant subgroup of $G$. 
\begin{itemize}
\item[(a)] The purification $H_*$ of $H$ in $G$ is $\f$-invariant and $h(\f\restriction_H)=h(\f\restriction_{H_*})$.
\item[(b)] If $H$ is essential in $G$, then $h(\phi)=h(\phi\restriction_H)$. Consequently, $\AT_h(G,\phi,H)$.
\end{itemize}
\end{corollary}
\begin{proof} Even if (b) follows from (a) with $H_* = G$, we shall prove first (b).  
Consider the common divisible hull $D$ of $H$ and $G$. Then $\widetilde \f:D\to D$ is the (unique)  common extension of $\f$ and $\f\restriction_H$. Proposition \ref{AA} applies to the pairs $D, G$ and $D, H$, giving $h(\f)=h(\widetilde\f)=h(\f\restriction_H)$. For the second assertion note that this equality  implies $\AT_h(G,\phi,H)$ in case  $h(\phi)=\infty$. If $h(\phi)<\infty$, then $h(\overline\phi)=0$ by Corollary \ref{CoAAA} and so $\AT_h(G,\phi,H)$ follows again.

Now we can deduce (a) from (b). The verification of the first assertion is immediate (see for example \cite[Lemma 3.3(a)]{SZ2}). The second assertion follows from (b) applied to $H_*$, $\f\restriction_{H_*}\in \End(H_*)$ and the $\f$-invariant subgroup $H$ of $H_*$. 
\end{proof}

The final part of item (b) of this corollary can be stated equivalently as $\AT_h(G,\phi,H)$ whenever $G$ is torsion-free and $G/H$ is torsion.

The next corollary is another consequence of Proposition \ref{AA} and Corollary \ref{CoAA}. It shows that the verification of the Addition Theorem for torsion-free abelian groups and their pure subgroups can be reduced to the case of divisible ones.

\begin{corollary}\label{->D(G)}
Let $G$ be a torsion-free abelian group, $\f\in\End(G)$ and $H$ a pure $\phi$-invariant subgroup of $G$. Considering $D(H)$ as a subgroup of $D(G)$, one has 
$$h(\phi)=h(\widetilde\phi), \ h(\phi\restriction_H)=h(\widetilde\phi\restriction_{D(H)})\ \mbox{ and }\ h(\overline {\f})=h(\overline {\widetilde \f}),$$ 
where $\overline\phi\in \End(G/H)$ and $\overline{\widetilde\phi}\in \End(D(G)/D(H))$. In particular, 
\begin{equation}\label{puro:in:divisibile}
\AT_h(G,\f,H) \Leftrightarrow \AT_h(D(G),\widetilde\f,D(H)).
\end{equation}
\end{corollary}
\begin{proof}
Since the purification $H_*$ of $H$ in $D(G)$ is divisible in view of Lemma \ref{div<->pure}(b),  $H_*=D(H)$ and $H = D(H) \cap G$, by the purity of $H$ in $G$.
By Proposition \ref{AA}, 
\begin{equation}\label{lig1}
h(\phi)=h(\widetilde\phi)\ \text{and}\ h(\phi\restriction_H)=h(\widetilde\phi\restriction_{D(H)}).
\end{equation}
Let $\pi:D(G)\to D(G)/D(H)$ be the canonical projection. Then $\pi(G)$ is essential in $D(G)/D(H)$ and $h(\overline{\widetilde\phi})=h(\overline{\widetilde\phi}\restriction_{\pi(G)})$ by Corollary \ref{CoAA}(b). Since $G/H\cong \pi(G)$, and $\overline\f$ is conjugated to $\overline{\widetilde\f}\restriction_{\pi(G)}$ by this isomorphism, $h(\overline\f)=h(\overline{\widetilde\f}\restriction_{\pi(G)})$ by Proposition \ref{properties}(a). Hence, $h(\overline {\f})=h(\overline {\widetilde \f})$. 
The conclusion follows from this equality and \eqref{lig1}.
\end{proof}

The validity of the right-hand-side of \eqref{puro:in:divisibile}, i.e., $\AT_h(D,\psi,K)$, for a \emph{finite-rank} divisible torsion-free abelian group $D$, $\psi\in\End(D)$ and a pure $\psi$-invariant subgroup $K$ of $D$, is established in Theorem \ref{AT-tf-pure} below. 

\medskip
 We conclude the section computing the algebraic entropy of the endomorphisms of $\Q$.

\begin{example}\label{Qex}
\begin{itemize}
\item[(a)] Let $\phi\in\End(\Q)$. Then there exists $r\in\Q$ such that $\f=\mu_r$.  If $r=0,\pm 1$, then $h(\phi)=0$ by Example \ref{h(id)=0} and \eqref{-f}.  By \eqref{-f} we can assume that $r>0$, and applying Proposition \ref{properties}(b) we may assume also that $r>1$. Let $r=\frac{a}{b}$, where $(a,b)=1$. We prove that $h(\phi)=\log a$. 

To prove that $h(\phi)\geq \log a$, take $F_0=\{0,1,\ldots,a-1\}$ and let $n\in\N_+$. 
It is easy to check that all sums $f_0+f_1\frac{a}{b}+\ldots+f_{n-1}\frac{a^{n-1}}{b^{n-1}}$, with $f_i\in F_0$, are pairwise distinct.
This shows that $|T_n(\phi,F_0)|=a^n$, and so $H(\phi,F_0)=\log a$, witnessing $h(\phi)\geq\log a$.

To prove the inequality $h(\phi)\leq\log a$, note that the subgroup $H$ of $\Q$ formed by all fractions having as denominators powers of $b$ (i.e., the subring of $\Q$ generated by $\frac{1}{b}$) is essential and $\phi$-invariant. Hence, $h(\phi)=h(\phi\restriction_H)$ by Corollary \ref{CoAA}(b).
Now for any $m\in\N_+$ consider $F_m=\left\{\pm\frac{r}{b^m}:  0\leq r \leq mb^m\right\}$. So $F_m=\hull{\frac{1}{b^m}}\cap [-m,m]$, where the 
interval $[-m,m]$ is taken in $H$. Let us observe that $\phi^k(F_m)\subseteq \hull{\frac{1}{b^{m+k}}}\cap \left[-m\frac{a^k}{b^k},m\frac{a^k}{b^k}\right]$, consequently $$T_n(\phi, F_m)\subseteq \underbrace{M+\ldots+ M}_n$$ where $M=\hull{\frac{1}{b^{m+n-1}}}\cap  \left[-m\frac{a^{n-1}}{b^{n-1}},m\frac{a^{n-1}}{b^{n-1}}\right]$. Therefore, 
$|T_n(\phi, F_m)|\leq 2nb^{m+n-1}m\frac{a^{n-1}}{b^{n-1}}$. Hence, 
\begin{equation*}\begin{split}
\log|T_n(\phi,F_m)|\leq \log2n+m\log b+(n-1)\log a.
\end{split}\end{equation*}
Thus $H(\phi,F_m)\leq\log a$. Since each $F\in[H]^{<\omega}$ is contained in $F_m$ for some $m\in\N_+$, this proves that $h(\phi)=h(\phi\restriction_H)\leq\log a$. 

\item[(b)] For $n\in\N_+$ and $r=\frac{a}{b}\in\Q$ with $a>b>0$, consider $\mu_r:\Q^n\to\Q^n$. 
Applying Proposition \ref{properties}(d) and item (a) we conclude that $h(\mu_r)=n\log a$. 

\item[(c)] Let $G$ be a torsion-free abelian group and consider $\mu_k:G\to G$ for some $k\in\N$ with $k>1$.
Then $$h(\mu_k)=\begin{cases}r(G)\log k & \text{if}\ r(G)\ \text{is finite},\\ \infty& \text{if}\ r(G)\ \text{is infinite}.\end{cases}$$
According to Proposition \ref{AA}, we can assume without loss of generality that $G$ is divisible. Let $\alpha=r(G)$. Then $G\cong\Q^{(\alpha)}$, with ${\mu_k}$ conjugated to $\mu_k^{\Q^{(\alpha)}}:\Q^{(\alpha)}\to\Q^{(\alpha)}$. 
If $\alpha\in\N$, then $h({\mu_k})=h(\mu^{\Q^{(\alpha)}}_k)=\alpha\log k$, by item (b) and Proposition \ref{properties}(a). If $\alpha$ is infinite, by Lemma \ref{restriction_quotient} and in view of the finite case, $h(\mu_k)> n\log k$ for every $n\in\N$. Hence, $h(\mu_k)=\infty$.  
\end{itemize}
\end{example}

In item (a) of the above example we have computed explicitly the algebraic entropy of $\mu_r:\Q\to \Q$, with $r=\frac{a}{b}>1$ and $(a,b)=1$. One can also apply the Algebraic Yuzvinski Formula; indeed, the unique eigenvalue of $\mu_r$ is $\frac{a}{b}>1$, and so the Algebraic Yuzvinski Formula \ref{Yuz} gives $h(\mu_r)=\log a$.
This formula was proved by Abramov in \cite{Ab} for the topological entropy of the automorphisms of $\widehat\Q$.

\section{Entropy of finitely generated flows}\label{fg-sec}

Consider the category $\abg$ of all abelian groups and their homomorphisms.  One can introduce the category $\af_\abg$ of flows of $\abg$ with objects the pairs $(G,\phi)$ with $G\in\abg$ and $\phi\in\End(G)$ (named \emph{algebraic flows} in \cite{DG1}).  A morphism $u:(G,\phi)\to(H,\psi)$ in $\af_\abg$ between two algebraic flows $(G,\phi)$ and $(H,\psi)$ is given by a homomorphism $u: G \to H$ such that the diagram
\begin{equation}\label{casc-mor}
\begin{CD}
G @>\phi>> G\\ 
@V u VV  @VV u V\\ 
H @>>\psi> H
\end{CD}
\end{equation}
in $\abg$ commutes. Two algebraic flows $(G,\phi)$ and $(H,\psi)$ are isomorphic in $\af_\abg$ precisely when the homomorphism $u: G \to H$ in \eqref{casc-mor}
is an isomorphism in $\abg$.

\smallskip
If $G$ is an abelian group and $\phi\in\End(G)$, then $G$ admits a structure of $\Z[t]$-module with multiplication determined by $tx=\phi(x)$ for every $x\in G$; we denote by $G_\phi$ the abelian group $G$ seen as a $\Z[t]$-module under the action of $\phi$. One has an isomorphism of categories
\begin{equation}\label{iso}
\af_\abg\cong\Mod_{\Z[t]}, 
\end{equation}
given by the functor $F:\af_\abg\to\Mod_{\Z[t]}$, associating to an algebraic flow $(G,\phi)$ the $\Z[t]$-module $G_\phi$ (see \cite[Theorem 3.2]{DG1}). In particular, for a morphism $u:(G,\phi)\to (H,\psi)$ in $\af_\abg$, $F(u)=u:G_\phi\to H_\psi$ is a homomorphism of $\Z[t]$-modules.
In the opposite direction consider the functor $F':\Mod_{\Z[t]}\to\af_\abg$, which associates to $M\in\Mod_{\Z[t]}$ the algebraic flow $(M,\mu_t)$, where
we denote still by $M$ the underling abelian group of the module $M$. If $u:M\to N$ is a homomorphism in $\Mod_{\Z[t]}$, then $F'(u)=u:(M,\mu_t)\to(N,\mu_t)$ is a morphism in $\af_\abg$.

\begin{remark}\label{af-mod}
By the isomorphism \eqref{iso}, every function $f:\af_\abg\to\R_{\geq0}\cup\{\infty\}$ defined on endomorphisms of abelian groups can be viewed as a function $f:\Mod_{\Z[t]}\to\R_{\geq0}\cup\{\infty\}$ by letting $f(M)=f(M,\mu_t)=f(\mu_t)$.

In particular, this holds for the algebraic entropy, so we can consider $$h:\Mod_{\Z[t]}\to\R_{\geq0}\cup\{\infty\},$$
letting $h(G_\phi)=h(\phi)$ for every $G_\phi\in\Mod_{\Z[t]}$.
\end{remark}

\begin{definition}[see Definition \ref{fg-def}]
Let $(G,\phi)$ be an algebraic flow. We say that $(G,\phi)$ is a \emph{finitely generated flow} if $G_\phi$ is a finitely generated $\Z[t]$-module (i.e., $G$ is finitely $\phi$-generated). 
\end{definition}

Clearly,  the supporting group of a finitely generated flow is countable. 

\medskip
Our aim is to reduce the computation of algebraic entropy to the case of finitely generated flows. We start from basic properties.

\begin{lemma}\label{fin-gen+noet->fin-gen}
Let $G$ be an abelian group, $\phi\in \End(G)$ and assume that $G$ is finitely $\phi$-generated.
\begin{itemize}
\item[(a)] If $H$ is a $\phi$-invariant subgroup of $G$, then $H$ is finitely $\phi\restriction_H$-generated and $G/H$ is finitely $\overline\phi$-generated.
\item[(b)] There exists $m\in\N_+$ such that $m \cdot t(G)=0$; so there exists a torsion-free subgroup $K$ of $G$ such that $G\cong K\times t(G)$.
\end{itemize}
\end{lemma}
\begin{proof}
(a) The first assertion follows from the fact that the ring $\Z[t]$ is Noetherian, the second one is obvious. 

\smallskip
(b) By (a) $t(G)$ is finitely generated as a $\Z[t]$-module, that is, $t(G)=V(\f,F')$ for some $F'\in[t(G)]^{<\omega}$. Since $F'$ is finite, there exists $m\in\N_+$ such that $mF'=0$. Then $m\cdot t(G)=0$. The second assertion follows from a theorem of Kulikov (see \cite[Section 27.5]{F}).
\end{proof}

For an abelian group $G$ and $\phi\in\End(G)$, let 
$$\mathfrak F(G,\phi):=\{V(\phi,F):F\in[G]^{<\omega}\}$$
be the family of all finitely $\phi$-generated subgroups of $G$. All these subgroups are clearly $\phi$-invariant.

\begin{lemma}\label{remark-G_F-gen}
Let $G$ be an abelian group and $\f\in\End(G)$. Then $G=\varinjlim\mathfrak F(G,\phi)$, and so
\begin{equation}\label{lastEq}
h(\f)=\sup \{h(\f\restriction_N): N\in\mathfrak F(G,\phi) \}.
\end{equation}
\end{lemma}
\begin{proof}
The family $\mathfrak F(G,\phi)$ is a direct system, as $N_1 + N_2 \in \mathfrak F(G,\phi)$ for $N_1,N_2\in \mathfrak F(G,\phi)$. That $h(\f)=\sup \{h(\f\restriction_N): N\in\mathfrak F(G,\phi) \}$ follows from Proposition \ref{properties}(c).
\end{proof}

Since $h$ is defined ``locally'', Lemma \ref{remark-G_F-gen} permits to reduce the computation of the algebraic entropy to finitely generated flows.

\smallskip
We see now that for every algebraic flow $(G,\phi)$ there exists a countable $\phi$-invariant subgroup $S$ of $G$ such that $h(\phi)=h(\phi\restriction_S)$.

\begin{definition}
Let $G$ be an abelian group and $\f\in \End(G)$. An \emph{entropy support} of $(G,\f)$ is a countable $\f$-invariant subgroup $S$ of $G$ such that $h(\f\restriction_S)=h(\f)$.
\end{definition}

Clearly, every countable $\f$-invariant subgroup of $G$ containing an entropy support of $(G,\phi)$ has the same property. This means that such a subgroup is not uniquely determined. 
The family ${\mathcal S}(G,\f)$ of all entropy supports of $(G,\f)$ is always non-empty:

\begin{lemma}\label{reduction-to-countable}
Let $G$ be an abelian group and $\f\in\End(G)$. Then there exists an entropy support of $(G,\phi)$.
\end{lemma}
\begin{proof}
By \eqref{lastEq}, there exists a subfamily $\{F_n\}_{n\in\N}\subseteq [G]^{<\omega}$ such that $h(\f)=\sup_{n\in\N}H(\f,F_n)$. Then $S=\sum_{n\in\N}V(\f, F_n)$ is a countable $\f$-invariant subgroup of $G$ such that 
$$h(\f)=\sup_{n\in\N}H(\f,F_n)=\sup_{n\in\N}H(\f\restriction_S,F_n)=h(\f\restriction_S).$$
Hence, $S\in\mathcal S(G,\phi)$.
\end{proof}

The next proposition reduces the proof of the Addition Theorem to the case of finitely $\phi$-generated abelian groups of finite rank. 

\begin{proposition}\label{sup->lim}
Let $G$ be an abelian group, $\f\in\End(G)$ and $H$ a $\f$-invariant subgroup of $G$. Denote by $\pi:G\to G/H$ the canonical projection. Then:
\begin{itemize}
\item[(a)] $\mathfrak F(H, \phi\restriction_H)=\{L \cap H:L\in\mathfrak F(G,\phi)\}$ and $\mathfrak F(G/H, \overline\phi)=\{\pi(L) :L\in\mathfrak F(G,\phi)\}$;
\item[(b)] there exists a chain $\{L_n\}_{n\in\N}$ in $\mathfrak F(G,\phi)$ such that
\begin{equation}\label{Nov.8}
h(\f) = \lim_{n\to \infty}h(\f\restriction_{L_n}),\ h(\f\restriction_H) = \lim_{n\to \infty}h(\f\restriction_{L_n\cap H})  \ \text{and} \ h(\overline\f)=\lim_{n\to\infty}h(\overline\f\restriction_{\pi(L_n)});
\end{equation}
\item[(c)] if $\AT_h(N,\f,N\cap H)$ for all $N\in \mathfrak F(G,\phi)$, then $\AT_h(G,\f,H)$;
\item[(d)] if $\AT_h(N,\f,N\cap H)$ for all $N\in \mathfrak F(G,\phi)$ of finite rank, then $\AT_h(G,\f,H)$. 
\end{itemize}
\end{proposition}
\begin{proof}
(a) The first equality follows from Lemma \ref{fin-gen+noet->fin-gen}(a) and the second is clear.

\smallskip
(b) By Lemma \ref{reduction-to-countable}, there exist $S\in\mathcal S(G,\f)$, $S_1\in\mathcal S(H,\phi\restriction_H)$ and $S_2\in\mathcal S(G/H,\overline\f)$. We can assume without loss of generality that $S\supseteq S_1$ and $\pi(S)\supseteq S_2$. Then $\pi(S)\in\mathcal S(G/H,\overline\phi)$. Moreover, $K :=\ker \pi\!\! \restriction_S = S \cap H \supseteq  S_1$, $K\in\mathcal S(H,\phi\restriction_H)$, and $K$ is a $\phi$-invariant subgroup of $S$ such that $S/K\cong\pi(S)$. Let $\overline{\phi\restriction_S}\in\End(S/K)$, which is conjugated to $\overline\phi\restriction_{\pi(S)}$. By hypothesis and by Proposition \ref{properties}(a),
\begin{equation}\label{--}
h(\phi) =h(\phi\restriction_S) ,\  h(\f\restriction_H) = h(\f\restriction_{K})\ \text{and}\ h(\overline\f)=h(\overline\f\restriction_{\pi(S)})= h(\overline{\phi\restriction_S}).
\end{equation}

Let $S=\{g_n:n\in\N\}$, and for every $n\in\N$ let $F_n=\{g_0,\ldots,g_n\}$. Then $S$ is increasing union of the subsets $F_n$ and consequently of the subgroups $L_n= V(\f,F_n)$, $n\in\N$. For every $F\in[S]^{<\omega}$ there exists $n\in\N$ such that $F\subseteq F_n$, hence $V(\phi,F)\subseteq L_n$ and this shows that the countable chain $\{L_n\}_{n\in\N}$ is cofinal in $\mathfrak F(S,\phi)$. Since $L_n\cap K=L_n\cap H$ for every $n\in\N$, it follows from (a) that $\{L_n\cap H\}_{n\in\N}$ is a countable cofinal chain in $\mathfrak F(K, \phi\restriction_{K})$; moreover, $\{\pi(L_n)\}_{n\in\N}$ is a countable cofinal chain in $\mathfrak F(\pi(S), \overline\phi\restriction_{\pi(S)})$.

By the cofinality of $\{L_n\}_{n\in\N}$ in $\mathfrak F(S,\phi)$, and by Lemma \ref{remark-G_F-gen}, one has $h(\f\restriction_S) = \sup_{n\in\N}h(\f\restriction_{L_n})$.  Since $\{h(\f\restriction_{L_n}):n\in\N\}$ is a non-decreasing sequence, this supremum becomes a limit $h(\f\restriction_S) = \lim_{n\to\infty}h(\f\restriction_{L_n})$.
Analogously, $h(\phi\restriction_{K})=\lim_{n\to\infty}h(\phi\restriction_{L_n\cap K})$ and $h(\overline \phi\restriction_{\pi(S)})=\lim_{n\to\infty}h(\phi\restriction_{\pi(L_n)})$. Now the required equalities follow from \eqref{--}.

\smallskip
(c) According to the hypothesis, for every $N\in\mathfrak F(G,\phi)$,
$$ h(\f\restriction_{N})=h(\f\restriction_{N\cap H})+h(\overline{\f\restriction_N}),$$ where $\overline{\f\restriction_{N}} \in \End(N/(N\cap H))$ is the induced endomorphism. In view of \eqref{restr-to-pi},
\begin{equation}\label{lig4}
h(\overline{\f\restriction_{N}})=h(\overline\f\restriction_{N/(N\cap H)}).
\end{equation}
By item (b), there exists a chain $\{L_n\}_{n\in\N}$ in $\mathfrak F(G,\phi)$ with \eqref{Nov.8}. Moreover, \eqref{lig4} (applied to $N=L_n$) entails $$h(\overline\f)=\lim_{n\to\infty}h(\overline\f\restriction_{\pi(L_n)})=\lim_{n\to\infty}h(\overline{\f\restriction_{L_n}}).$$
As $L_n \in \mathfrak F(G,\phi)$, one has 
$$h(\f)=\lim_{n\to\infty}h(\f\restriction_{L_n})=\lim_{n\to \infty}h(\f\restriction_{L_n\cap H})+\lim_{n\to \infty}h(\overline{\f\restriction_{L_n}})=h(\f\restriction_{H})+h(\overline\f),$$ that is, $\AT_h(G,\phi,H)$.

\smallskip
(d) If $r(V(\f,g))$ is infinite for some $g\in G$, then $\AT_h(G,\f,H))$ by Lemma \ref{CorollaryNov8:1}. Hence, we can assume that $r(V(\f,g))$ is finite for all $g\in G$. Then $r(N)$ is finite for all $N\in \mathfrak F(G,\phi)$. Hence, our hypothesis entails $\AT_h(N,\f,N\cap H))$ for all $N\in \mathfrak F(G,\phi)$. Now (c) yields $\AT_h(G,\f,H)$. 
\end{proof}

\section{The Addition Theorem}\label{AT-sec}

The aim of this section is to give a complete proof of the Addition Theorem.

\subsection{Skew products}\label{3-sec}

Let $K$ and $H$ be abelian groups, and let $\phi_1\in \End(K)$ and $\phi_2\in \End(H)$. The direct product $\pi = \phi_1 \times \phi_2: K\times H \to K \times H$ is defined by $\pi(x,y) = (\phi_1(x), \phi_2(y))$ for every pair $(x,y) \in K\times H$. For a homomorphism $s: K \to H$, the \emph{skew product of $\phi_1$ and $\phi_2$ via} $s$ is $\phi\in \End(K\times H)$ defined by 
\begin{equation*}
\phi(x,y)=(\phi_1(x),\phi_2(y)+s(x))\mbox{ for every }(x,y)\in K\times H.
\end{equation*} 
We say that the homomorphism $s$ is \emph{associated} to the skew product $\phi$.

When $s=0$ one obtains the usual direct product endomorphism $\phi = \pi = \phi_1\times \phi_2$. 

\smallskip
Identifying $H$ with the  $\phi$-invariant subgroup $0\times H$ of $K\times H$, the endomorphism induced by $\phi$ on $K \cong (K \times H)/H$ is precisely $\phi_1$. Let us see that the skew products arise precisely in such a circumstance.

\begin{remark}\label{rem-skew}
If $G$ is an abelian group and $\f\in\End(G)$, suppose to have a $\f$-invariant subgroup $H$ of $G$ that splits as a direct summand, that is $G=K\times H$ for some subgroup $K$ of $G$. Let us see that $\f$ is a skew product. Indeed, let $\iota: G/H\to K $ be the natural isomorphism and let $\overline\f\in \End(G/H)$. Denote by $\f_1:K \to K$ the endomorphism $\phi_1=\iota\overline\phi\iota^{-1}$ of $K$, and let $\f_2=\f\restriction_H$. It follows from the definition of $\f_1$ that  for every $x\in K$ there exists an element $s_\f(x)\in H$ such that $\f(x,0) - (\f_1(x), 0) = (0, s_\phi(x))$. The map $s_\f:K\to H$ is a homomorphism with  $\f(x,0) = (\f_1(x), s_\phi(x))$ for every $x\in K$. Hence, $\f(x,y)=(\f_1(x),\f_2(y)+s_\f(x))$ for every $(x,y)\in G=K\times H$. Therefore, $\f$ is the skew product of $\f_1$ and $\f_2$ via $s_\f$.
\end{remark}

As the next example shows, a natural instance to this effect are the fully invariant subgroups. 

\begin{example}\label{skewex}
 If $G$ is a abelian group and $\f\in\End(G)$, then $t(G)$ is fully invariant, so necessarily $\f$-invariant. 
According to Remark \ref{rem-skew}, $\f$ is a skew product when $t(G)$ splits in $G$, that is, $G = K \times t(G)$ where $K$ is a torsion-free subgroup of $G$.
\end{example}

In the sequel, for a skew product $\phi: G=K\times H \to K \times H$, we denote by $\phi_1:K\to K$ the endomorphism of $K$ conjugated to $\overline\phi$, and we let $\phi_2=\phi\restriction_H$. 
We refer to $\pi_\phi = \phi_1 \times \phi_2$ as the \emph{direct product associated to the skew product} $\phi$. We can extend to $G=K\times H$ the homomorphism $s_\phi:K\to H$ associated to the skew product, defining it to be $0$ on $H$. This allows us to consider $s_\phi\in \End(G\times H)$ and to speak of the composition $s_\phi^2 =0$, as well as $\phi=\pi_\phi+s_\phi$ in the ring $\End(G)$. In other words, the difference $s_\phi=\phi-\pi_\phi$ measures how much the skew product $\phi$ fails to coincide with its associated direct product $\pi_\phi$. 

\begin{proposition}\label{AT-semi-poor}
Let $G$ be an abelian group and $\f\in\End(G)$. Assume that $G=K\times T$ for some subgroups $K$ and $T$ of $G$, with $T$ torsion and $\f$-invariant, and suppose that $\phi$ is a skew product such that $s_\f(K)$ is finite. Then:
\begin{itemize}
\item[(a)] for every $n\in\N_+$ we have $T_n(\pi_\f,F_1\times F_2)= T_n(\f,F_1\times F_2)$ for every $F_1\in[K]^{<\omega}$ and $F_2\in[T]^{<\omega}$ with $F_2$ a subgroup of $T$ containing $s_\phi(K)$;
\item[(b)] $h(\f)=h(\pi_\f)$, consequently $\AT_h(G,\f,T)$.
\end{itemize}
\end{proposition}
\begin{proof}
(a) We have $\pi_\f^n(F_1\times F_2)=\f_1^n(F_1)\times \f_2^n(F_2)$ and so $T_n(\pi_\f,F_1\times F_2)=T_n(\f_1,F_1)\times T_n(\f_2,F_2)$ for every $n\in\N_+$. 

One can prove by induction that, for every $x\in K$ and every $n\in\N_+$,
\begin{equation*}
\f^n(x,0)=(\f_1^n(x),\f_2^{n-1}(s_\f(x))+\f_2^{n-2}(s_\f(\f_1(x)))+\ldots+\f_2(s_\f(\f_1^{n-2}(x)))+s_\f(\f_1^{n-1}(x))).
\end{equation*}
Since $s_\phi(K) \subseteq F_2$, we conclude that 
\begin{equation*}
\f^n(x,0)\in (\f_1^n(x), 0) + (0\times(\f_2^{n-1}(F_2)+\f_2^{n-2}(F_2)+\ldots+\f_2(F_2)+F_2))=
(\f_1^n(x), 0)+ (0\times T_n(\f_2,F_2));
\end{equation*}
as $0\times T_n(\f_2,F_2)= T_n(\f,0\times F_2)$ is a subgroup of $G$, we deduce that
\begin{equation}\label{eq-NEW2}
\f^n(x,0)\in (\f_1^n(x), 0)+ (0\times T_n(\f_2,F_2)) \mbox{ and }  (\f_1^n(x), 0) \in \f^n(x,0) + T_n(\f,0\times F_2). 
\end{equation}
Fix $m\in\N$ and an $m$-tuple $a_0, a_1, \ldots, a_{m-1}\in F_1$. Applying the first and the second part of \eqref{eq-NEW2} to $a_n$ for $n=0,1, \ldots , m-1$, we get respectively
\begin{equation}\label{svito1}
\sum_{n=0}^{m-1}\f^n(a_n,0)\in \sum_{n=0}^{m-1}(\f_1^n(a_n), 0)+(0\times T_m(\f_2,F_2)) \subseteq T_m(\f_1,F_1)\times T_m(\f_2,F_2)
\end{equation}
and 
\begin{equation}\label{svito2}
\sum_{n=0}^{m-1}(\f_1^n(a_n), 0) \in \sum_{n=0}^{m-1} \f^n(a_n,0) + T_m(\f,0\times F_2) \subseteq T_m(\f,F_1\times0) + T_m(\f,0\times F_2). 
\end{equation}
In other words, \eqref{svito1} and \eqref{svito2} yield
\begin{equation}\label{svito3}
T_m(\f,F_1\times 0) \subseteq T_m(\f_1,F_1)\times T_m(\f_2,F_2)   \mbox{ and } 
T_m(\f_1,F_1)\times 0 \subseteq T_m(\f,F_1\times 0) + T_m(\f,0\times F_2).
\end{equation}
As $0\times T_m(\f_2,F_2)= T_m(\f,0\times F_2)$ is a subgroup of $G$, the two parts of \eqref{svito3} give respectively
\begin{equation*}
T_m(\f,F_1\times F_2)= T_m(\f,F_1\times 0) + T_m(\f,0\times F_2) \subseteq T_m(\f_1,F_1)\times T_m(\f_2,F_2)  = T_m(\pi_\f,F_1\times F_2)
\end{equation*}
and
\begin{equation*}
T_m(\pi_\f,F_1\times F_2) = T_m(\f_1,F_1)\times T_m(\f_2,F_2) \subseteq T_m(\f,F_1\times 0) + T_m(\f,0\times F_2) =T_m(\f,F_1\times F_2).
\end{equation*}
These two inclusions prove the required equality $T_n(\pi_\f,F_1\times F_2)= T_n(\f,F_1\times F_2)$.

\smallskip
(b) Let $F\in[G]^{<\omega}$. Then $F\subseteq F_1\times F_2$, for some $F_1\in[K]^{<\omega}$ and $F_2\in[T]^{<\omega}$. By Remark \ref{0inF}(b), we can assume without loss of generality that $(0,0)\in F_1\times F_2$ and that $F_2$ is a subgroup of $T$ with $F_2\supseteq s_\phi(K)$. The conclusion follows from item (a)  and Remark \ref{0inF}(b).

To prove the second assertion of (b), consider $\overline\f\in \End(G/T)$.
By Proposition \ref{properties}(d), $h(\pi_\f)=h(\f_1)+h(\f_2)$. Moreover, by definition, $\phi_2=\phi\restriction_T$ and $\phi_1$ is conjugated to $\overline\phi$, so $h(\overline\phi)=h(\phi_1)$ by Proposition \ref{properties}(a). 
\end{proof}

Now we show that the Addition Theorem holds with respect to the torsion subgroup.

\begin{proposition}\label{AT-t(G)}
Let $G$ be an abelian group and $\f\in\End(G)$. Then $\AT_h(G,\f,t(G))$.
\end{proposition}
\begin{proof} 
By Proposition \ref{sup->lim}(d), we can assume that $r(G)$ is finite and $G$ is finitely $\phi$-generated.
By Lemma \ref{fin-gen+noet->fin-gen}(b) there exists $m\in\N_+$ such that $m\cdot t(G)=0$ and there exists a torsion-free subgroup $K$ of $G$ such that $G\cong K\times t(G)$.
Since $t(G)$ is a $\phi$-invariant subgroup of $G$ that splits, by Remark \ref{rem-skew} this gives rise to a skew product, that is, there exists a homomorphism $s_\f: K\to t(G)$ such that $\f(x,y)=(\f_1(x),\f_2(y)+s_\f(x))$ for every $(x,y)\in G$; moreover, $\f_1\in \End(K)$ is conjugated to $\overline\f \in \End(G/t(G))$ by the isomorphism $K\cong G/t(G)$, and $\f_2=\f\restriction_{t(G)}$. 

To show that $s_\f(K)$ is finite, we notice first that $K/mK$ is finite by \cite[Theorem 0.1]{A}, since $K$ is torsion-free of finite rank. As  $s_\phi(mK)=m s_\phi(K)\subseteq m t(G)=0$, $s_\phi$  factorizes through the canonical projection $\pi:K\to K/mK$, i.e., there exists a homomorphism $\psi:K/mK\to t(G)$ such that $s_\phi=\psi\circ\pi$.  Therefore, $s_\phi(K)=\psi(K/mK)$ is finite.

Now Proposition \ref{AT-semi-poor} applies and gives $\AT_h(G,\phi,t(G))$.
\end{proof}

\subsection{Addition Theorem with respect to two invariant subgroups}\label{rel-sec}

In the sequel $\wedge$ stays for conjunction.

\begin{proposition}\label{6carte}
Let $G$ be an abelian group, $\f\in\End(G)$ and $H,K$ two $\f$-invariant subgroups of $G$. Then
$$\AT_h(G,\f,K) \wedge \AT_h(H,\f\restriction_H,H\cap K) \wedge \AT_h(K,\f\restriction_K,H\cap K) \wedge $$
$$ \wedge \AT_h(G/H,\overline\f_H,(H+K)/H) \wedge \AT_h(G/K,\overline\f_K,(H+K)/K) \Longrightarrow \AT_h(G,\f,H).$$
\end{proposition}
\begin{proof}
The situation is described by the following diagrams involving the triples $(G,\phi,H)$ and $(G/K,\overline\phi_K, (H+K)/K)$.
\begin{equation*}
\xymatrix@!R{
*+++{H} \ar[r]^{\phi\restriction_H} \ar@{>->}[d] & *+++{H}\ar@{>->}[d] & & *++{(H+K)/K}\ar[r]^{\overline\f_K\restriction_{(H+K)/K}}\ar@{>->}[d] & *++{(H+K)/K}\ar@{>->}[d] \\
G \ar[r]^{\phi} \ar@{->>}[d] & G \ar@{->>}[d] & & G/K \ar[r]^{\overline\f_K} \ar@{->>}[d] & G/K \ar@{->>}[d] \\
G/H \ar[r]^{\overline\f_H} & G/H & & (G/K)/((H+K)/K) \ar[r]^{\overline{\overline\f_K}} & (G/K)/((H+K)/K)
}
\end{equation*}
Our hypotheses imply that:
\begin{itemize}
\item[(i)]$h(\f)=h(\f\restriction_K)+h(\overline\f_K)$;
\item[(ii)]$h(\f\restriction_H)=h(\f\restriction_{H\cap K})+h(\overline{\f\restriction_H})$;
\item[(iii)]$h(\f\restriction_K)=h(\f\restriction_{H\cap K})+h(\overline{\f\restriction_K})$;
\item[(iv)]$h(\overline\f_H)=h(\overline\f_H\restriction_{(H+K)/H})+h(\overline{\overline\f_H})$;
\item[(v)]$h(\overline\f_K)=h(\overline\f_K\restriction_{(H+K)/K})+h(\overline{\overline\f_K})$.
\end{itemize}
The composition of the isomorphisms $(G/K)/((H+K)/K)\cong G/(H+K) \cong (G/H)/((H+K)/H)$ commutes with 
$$
\overline{\overline\f_K}:(G/K)/((H+K)/K)\to(G/K)/((H+K)/K)\mbox{ and }\overline{\overline\f_H}:(G/H)/((H+K)/H)\to (G/H)/((H+K)/H),
$$ i.e., $\overline{\overline\f_K}$ and $\overline{\overline\f_H}$ are conjugated. So Proposition \ref{properties}(a) yields
\begin{equation}\label{hbar=hbar}
h(\overline{\overline\f_H})=h(\overline{\overline\f_K}),
\end{equation}
Moreover, by \eqref{restr-to-pi}, we have
\begin{equation}\label{hover=hbar}
h(\overline{\f\restriction_H})=h(\overline\f_K\restriction_{H+K/K})\ \text{and}\ h(\overline{\f\restriction_K})=h(\overline\f_H\restriction_{H+K/H}).
\end{equation}
From (i), (iii), (v), \eqref{hbar=hbar}, \eqref{hover=hbar}, (ii) and (iv), we deduce 
\begin{align*}
h(\f)&=h(\f\restriction_K)+h(\overline\f_K)\\
&=(h(\f\restriction_{H\cap K})+h(\overline{\f\restriction_K}))+(h(\overline\f_K\restriction_{(H+K)/K})+h(\overline{\overline\f_K}))\\
&=h(\f\restriction_{H\cap K})+h(\overline\f_H\restriction_{H+K/H})+h(\overline{\f\restriction_H})+h(\overline{\overline\f_H})\\
&=(h(\f\restriction_{H\cap K})+h(\overline{\f\restriction_H}))+(h(\overline\f_H\restriction_{H+K/H})+h(\overline{\overline\f_H}))\\
&=h(\f\restriction_H)+h(\overline\f_H).
\end{align*}
Therefore, $h(\f)=h(\f\restriction_H)+h(\overline\f_H)$.
\end{proof}

The following corollary can be deduced easily from Proposition \ref{6carte}.

\begin{corollary}\label{3carte}
Let $G$ be an abelian group, $\f\in\End(G)$ and $H, K$ two $\f$-invariant subgroups of $G$ with $H\subseteq K$. Then:
\begin{itemize}
\item[(a)]$\AT_h(G,\f,H) \wedge \AT_h(K,\f\restriction_K,H)\wedge \AT_h(G/H,\overline \f, K/H) \Longrightarrow \AT_h(G,\f,K)$; and
\item[(b)]$\AT_h(G,\f,K) \wedge \AT_h(K,\f\restriction_K,H)\wedge \AT_h(G/H,\overline \f, K/H) \Longrightarrow \AT_h(G,\f,H)$.
\end{itemize}
\end{corollary}

Proposition \ref{6carte} and Corollary \ref{3carte} have the following useful consequences, that are applied below in the proof of the Addition Theorem.

\begin{lemma}\label{lem1.9}\label{cor0}
Let $G$ be an abelian group, $\f\in\End(G)$ and $H$ a $\f$-invariant subgroup of $G$. Then $\AT_h(G,\phi,H)$ if one of the following conditions holds:
\begin{itemize}
\item[(a)] $G/H$ is torsion;
\item[(b)] $H$ is torsion.
\end{itemize}
\end{lemma}
\begin{proof} 
(a) Our aim is to apply Proposition \ref{6carte} with $K=t(G)$, so we check that its hypotheses are satisfied: 
\begin{itemize}
\item[(i)] $\AT_h(G,\f,t(G))$ and $\AT_h(H,\f\restriction_H,t(H))$, according to Proposition \ref{AT-t(G)}; 
\item[(ii)] $\AT_h(t(G),\f\restriction_{t(G)},t(H))$ and  $\AT_h(G/H,\overline\f_H,(H+t(G))/H)$, because $t(G)$ and $G/H$ are torsion, so
  the Addition Theorem for torsion abelian groups from \cite{DGSZ} applies;
\item[(iii)] $\AT_h(G/t(G),\overline\f_{t(G)},(H+t(G))/t(G))$ by Corollary \ref{AT-tf-ess}, as $G/t(G)$ is torsion-free and $(H+t(G))/t(G)$ is essential in $G/t(G)$, being $G/(H+t(G))$ torsion as a quotient of the torsion abelian group $G/H$.
\end{itemize}
As $t(H) = H \cap t(G)$,  Proposition \ref{6carte} applies to conclude the proof.

\smallskip
(b) The subgroup $t(G)/H$ of $G/H$ is precisely $t(G/H)$, so both $\AT_h(G/H,\overline \f, t(G)/H)$ and $\AT_h(G,\f, t(G))$ by Proposition \ref{AT-t(G)}. On the other hand, $\AT_h(t(G),\f\restriction_{t(G)}, H)$ as $t(G)$ is torsion (apply the Addition Theorem for the torsion abelian groups from \cite{DGSZ}). Now Corollary \ref{3carte}(b) applies to the triple $H\subseteq t(G)\subseteq G$.
\end{proof}

\subsection{The torsion-free case}\label{sec-tf}

If the abelian group $G$ is torsion-free, then for any $\phi\in\End(G)$ the subgroup $\ker_\infty \f$ is pure in $G$.
The next result reduces the computation of the algebraic entropy of endomorphisms of finite-rank torsion-free abelian groups to the case of \emph{injective} ones.

\begin{proposition}\label{red-to-inj}
Let $G$ be a torsion-free abelian group of finite rank and $\f\in\End(G)$. Then $h(\f) = h(\overline\f_{\ker_\infty\f})$. Consequently, $\AT_h(G,\f,\ker_\infty\f)$.
\end{proposition}
\begin{proof}
We prove that $h(\f) = h(\overline\f)$, then this implies immediately $\AT_h(G,\f,\ker_\infty\f)$, since $h(\phi\restriction_{\ker_\infty\phi})=0$ by Proposition \ref{ker_infty}.

Suppose first that $G$ is divisible. Since $G$ has finite rank, $\ker_\infty\f$ has finite rank as well. Then there exists $n\in\N$, such that $\ker_\infty\f=\ker\f^n$.
Let $\gamma=\f^n$ and $\overline\gamma\in\End(G/\ker\gamma)$ the endomorphism induced by $\gamma$. Then $h(\gamma)=n h(\f)$ by Proposition \ref{properties}(b). Since $\overline\gamma=\overline\f^n$, it follows that $h(\overline\gamma)=n h(\overline\f)$ again by Proposition \ref{properties}(b). So, if we prove that $h(\gamma)=h(\overline\gamma)$, it follows that $h(\f)=h(\overline\f)$. Note that $\overline\gamma^2(G) = \overline\gamma(G)$. 

This shows that we can suppose without loss of generality that $\ker_\infty\f=\ker\f$ and $\f^2(G) = \f(G)$. Let $\overline\phi\in \End(G/\ker\phi)$. From \cite[Section 58, Theorem 1]{Halmos2} it follows that  
\begin{equation*}
G\cong\ker\f\times\f(G).
\end{equation*}
Proposition \ref{ker_infty} gives $h(\f\restriction_{\ker\f})=0$ and so $h(\phi)=h(\f\restriction_{\f(G)})$ by Proposition \ref{properties}(d). Since $\f(G)\cong G/\ker\f$, and $\f\restriction_{\f(G)}$ and $\overline\f$ are conjugated by this isomorphism, $h(\f\restriction_{\f(G)})=h(\overline\f)$ by Proposition \ref{properties}(a). Hence, $h(\phi)=h(\overline\phi)$.

\smallskip
We consider now the general case. Since $\ker_\infty\widetilde\phi$ is pure in $D(G)$, it is divisible by Lemma \ref{div<->pure}(b). Moreover, $\ker_\infty\phi$ is essential in $\ker_\infty\widetilde\f$. Indeed, let $x\in\ker_\infty\widetilde\f$, i.e., there exists $n\in\N_+$ such that ${\widetilde\f}^n(x)=0$. Since $G$ is essential in $D(G)$ there exists $k\in\Z$ such that $kx\in G\setminus\{0\}$. Furthermore, $\phi^n(kx)={\widetilde\phi}^n(kx)=k{\widetilde\phi}^n(x)=0$ and so $kx\in\ker_\infty\phi\setminus\{0\}$. It follows that $\ker_\infty \widetilde\phi=D(\ker_\infty\phi)$.
By the first part of the proof, $h(\widetilde\phi)=h(\overline{\widetilde\phi})$, where $\overline{\widetilde\phi}\in \End(D(G)/\ker_\infty\widetilde\phi)$ is the induced endomorphism. By Proposition \ref{AA} $h(\phi)=h(\widetilde\phi)$ and by Corollary \ref{->D(G)} $h(\overline\phi)=h(\overline{\widetilde\phi})$. Hence, $h(\phi)=h(\overline\phi)$.
\end{proof}

Our next aim is to prove the Addition Theorem for torsion-free abelian groups (see Proposition \ref{pure:subgroup}).
We start from the divisible case in Theorem \ref{AT-tf-pure}, where we apply the Algebraic Yuzvinski Formula.

\begin{theorem}\label{AT-tf-pure} 
Let $n\in\N_+$, $\phi\in \End(\Q^n)$ and let $H$ be a pure $\phi$-invariant subgroup of $\Q^n$. Then $h(\phi)<\infty$ and $\AT_h(\Q^n,\phi,H)$.
\end{theorem}
\begin{proof}
Let $D=\Q^n$. Theorem \ref{Yuz}  implies $h(\phi)<\infty$.

\medskip
Assume first that $\phi$ is an automorphism. Let $r(H) = k\in\N$, that is, $H \cong \Q^k$. Let $\mathcal B=\{v_1, \ldots, v_k, v_{k+1}, \ldots, v_n\}$ be a basis of $D$ over $\Q$ such that $\mathcal B_H=\{v_1, \ldots, v_k\}$ is a basis of $H$ over $\Q$. Then the matrix of $\f$ with respect to $\mathcal B$ has the following block-wise form: 
$$
A =\begin{pmatrix} A_1 & B \\ 0 & A_2\end{pmatrix},
$$ 
where $A_1$ is the matrix of $\f\restriction_H$ with respect to $\mathcal B_H$. Let $\pi:D\to D/H$ be the canonical projection and $\overline\phi\in \End(D/H)$ the endomorphism induced by $\phi$. Then $\overline{\mathcal B}=\{\pi(v_{k+1}), \ldots, \pi(v_n)\}$ is a basis of $D/H\cong \Q^{n-k}$ and $A_2$ is the matrix of $\overline\phi$ with respect to $\overline{\mathcal B}$. Let $\alpha_1,\ldots,\alpha_k$ be the eigenvalues of $A_1$ and let $\alpha_{k+1},\ldots,\alpha_{n}$ be the eigenvalues of $A_2$. Then $\alpha_1,\ldots,\alpha_n$ are the eigenvalues of $A$. 

Let $\chi$ and $\chi_1,\chi_2\in\Q[t]$ be the characteristic polynomials of $A$ and $A_1,A_2$ respectively. Then $\chi=\chi_1\chi_2$. 
Let $s_1$ and $s_2$ be the least common multiples of the denominators of the coefficients of $\chi_1$ and $\chi_2$ respectively. This means that $p_1=s_1\chi_1$ and 
$p_2=s_2\chi_2\in\Z[t]$ are primitive. By Gauss Lemma $p= p_1p_2$ is primitive and so for $s=s_1s_2$ the polynomial $p=s\chi \in\Z[t]$ is primitive.
Now the Algebraic Yuzvinski Formula \ref{Yuz} applied to $\phi,\phi\restriction_H,\overline\phi$ gives 
\begin{align*}
h(\f)&= \log s+\sum_{1\leq i\leq n,|\alpha_i|>1}\log|\alpha_i|\\
&=\log(s_1s_2)+\sum_{1\leq i\leq k,|\alpha_i|>1}\log|\alpha_i|+\sum_{k+1\leq i\leq n,|\alpha_i|>1}\log|\alpha_i|\\
&=\left(\log s_1+\sum_{1\leq i\leq k,|\alpha_i|>1}\log|\alpha_i|\right)+\left(\log s_2+\sum_{k+1\leq i\leq n,|\alpha_i|>1}\log|\alpha_i|\right)\\
&=h(\f\restriction_H) + h(\overline \f).
\end{align*}
This proves $\AT_h(\Q^n,\phi,H)$.

\medskip
Consider now the general case of $\phi\in\End(G)$ and let $\pi:D\to D/\ker_\infty\f$ be the canonical projection. Then $\pi(H)=(H+\ker_\infty\f)/\ker_\infty\f$ is a $\overline\f$-invariant pure (i.e., divisible) subgroup of $D/\ker_\infty\f$ and we have the following two diagrams.
\begin{equation*}
\xymatrix@!R{
*+++{H} \ar[r]^{\phi\restriction_H} \ar@{>->}[d] & *+++{H}\ar@{>->}[d] & & *+++{\pi(H)}\ar[r]^{\overline\f\restriction_{\pi(H)}}\ar@{>->}[d] & *+++{\pi(H)}\ar@{>->}[d] \\
D \ar[r]^{\phi} \ar@{->>}[d] & D \ar@{->>}[d] & & D/\ker_\infty\f \ar@{->}[r]^{\overline\f} \ar@{->>}[d] & D/\ker_\infty \f \ar@{->>}[d] \\
D/H \ar[r]^{\overline\f_H} & D/H & & (D/\ker_\infty\f)/\pi(H) \ar[r]^{\overline{\overline\f_{\pi(H)}}} & (D/\ker_\infty\f)/\pi(H)
}
\end{equation*}

To establish $\AT_h(D,\f,H)$ we intend to apply  Proposition \ref{6carte}. To this end  we check in the sequel the validity of its hypotheses:
\begin{itemize}
 \item[(i)] $\AT_h(D,\f,\ker_\infty\f)$, $\AT_h(H,\f\restriction_H,H\cap \ker_\infty\f)$ and $\AT_h(D/H,\overline\f_H,(H+\ker_\infty\f)/H)$ by Proposition \ref{red-to-inj}, as   $H\cap\ker_\infty\f  =\ker_\infty\f\restriction_H$ and $(H+\ker_\infty\f)/H = \ker_\infty\overline\f$; 
 \item[(ii)] $\AT_h(\ker_\infty\f,\f\restriction_{\ker_\infty\f},H\cap \ker_\infty\f)$, as $h(\f\restriction_{\ker_\infty\f})=0$ by Proposition \ref{ker_infty};
 \item[(iii)] $\AT_h(D/\ker_\infty\f,\overline\f,(H+\ker_\infty\f)/\ker_\infty\f)$ by the the first part of the proof (as $\overline\phi$ is bijective, being injective).
\end{itemize}
Now Proposition \ref{6carte} yields that $\AT_h(D,\f,H)$.
\end{proof}

The next result settles the torsion-free case of the Addition Theorem.
 
\begin{proposition}\label{tor_free}\label{pure:subgroup} Let $G$ be a torsion-free abelian group, $\f\in\End(G)$ and let $H$ be a $\phi$-invariant subgroup of $G$. Then $\AT_h(G,\f,H)$.
\end{proposition}
\begin{proof} 
By Proposition \ref{sup->lim}(d), we can assume that $G$ has finite rank. 
The purification $H_*$ of $H$ is $\f$-invariant too. We are going to deduce $\AT_h(G,\f,H)$ from Corollary \ref{3carte}(b) applied to the chain $H \subseteq H_* \subseteq G$. 
The validity of $\AT_h(G,\f,H_*)$ is granted by Theorem \ref{AT-tf-pure} and Corollary \ref{->D(G)}. 
Since $H$ is an essential subgroup of $H_*$, Corollary \ref{CoAA}(a) yields $\AT_h(H_*,\f\restriction_{H_*},H)$. Finally, $\AT_h(G/H,\overline \f,H_*/H)$ follows from Proposition \ref{AT-t(G)}, as $H_*/H= t(G/H)$. 
\end{proof}

Now we can prove the Addition Theorem:

\begin{proof}[\bf Proof of Theorem \ref{AT}]
One has $\AT_h(G, \f, t(G))$ and  $\AT_h(H+t(G), \f \restriction_{H+t(G)}, t(G))$ by Proposition \ref{AT-t(G)}, and $\AT_h(G/t(G), \overline \f, H+t(G)/t(G))$ by Proposition \ref{tor_free}. 
Then $$\AT_h(G,\f,H+t(G))$$ by Corollary \ref{3carte}(a) applied to the triple $t(G)\subseteq H + t(G)\subseteq G$.  
Since $H+t(G)/H$ is torsion, Lemma \ref{lem1.9}(a) yields $$\AT_h(H+t(G), \f \restriction_{H+t(G)}, H).$$
Finally, as the subgroup $(H+t(G))/H$ of $G/H$ is torsion, by Lemma \ref{cor0}(b) $$\AT_h(G/H, \overline \f, (H+t(G))/H).$$
Now Corollary \ref{3carte}(b) applies to the triple $H\subseteq H+t(G)\subseteq G$ to conclude the proof.
\end{proof}

\section{The Uniqueness Theorem}\label{uniq-sec}

We start this section by proving the Uniqueness Theorem for the algebraic entropy $h$ in the category of all abelian groups.
In other words, we have to show that, whenever a collection $h^*=\{h^*_G:G\ \text{abelian group}\}$ of functions $h^*_G:\End(G)\to\R_{\geq0}\cup\{\infty\}$ satisfies (a)--(e) of Theorem \ref{UT}, then $h^*(\phi)=h(\phi)$ for every abelian group $G$ and every 
$\f\in\End(G)$. 

\smallskip
We shall provide two proofs of this fact, but first we need to point out three easy consequences of the hypotheses  (a)--(e) of Theorem \ref{UT} to be used in both proofs. 
\begin{itemize}
\item[(A)] Item (c) implies that $h^*$ is monotone with respect to taking restrictions to invariant subgroups and to taking induced endomorphisms of the quotient groups with respect to invariant subgroups. 
\item[(B)] Item (d) says that $h^*$ and $h$ coincide on all Bernoulli shifts $\beta_K$, where $K$ is a finite abelian group. This can be extended also to $\beta_\Z$, by showing that $h^*(\beta_\Z)=\infty$ as follows. Let $G=\Z^{(\N)}$. 
For every prime $p$, the subgroup $pG$ of $G$ is $\beta_\Z$-invariant, so $\beta_\Z$ induces an endomorphism $\overline{\beta_\Z}:G/pG\to G/pG$. Since $G/pG\cong \Z(p)^{(\N)}$ and $\overline{\beta_\Z}$ is conjugated to $\beta_{\Z(p)}$ through this isomorphism, $h^*(\overline{\beta_\Z})=h^*(\beta_{\Z(p)})=\log p$ by item (a). Therefore, $h^*(\beta_\Z)\geq\log p$ for every prime $p$ by (A), and so $h^*(\beta_\Z)=\infty$.
\item[(C)] Item (e) means that $h^*$ and $h$ coincide on all endomorphisms of $\Q^n$, when $n\in \N_+$ varies. 
\end{itemize}

Roughly speaking, (B) and (C) (that are (d) and (e) of Theorem \ref{UT}, respectively) ensure the coincidence of $h$ and $h^*$ on all endomorphisms of $\Q^n$ for every $n\in \N_+$ and all Bernoulli shifts. We have to show that along with the properties (a), (b) and (c), this forces $h^*$ to coincide with $h$ on all endomorphisms of all abelian groups. 

\begin{proof}[\bf Proof of Theorem \ref{UT}] 
In the sequel $G$ is an arbitrary abelian group, $\f\in \End(G)$ and $H$ a $\f$-invariant subgroup of $G$. To prove that $h^*(\phi)=h(\phi)$ we consider various cases
depending on the group $G$. 

\smallskip
(i) If $G$ is torsion, then $h^*(\phi)=\ent(\phi)=h(\phi)$ by the Uniqueness Theorem for $\ent$ \cite[Theorem 6.1]{DGSZ}. Obviously, item (e) becomes vacuous in this case.

\smallskip
(ii)  It suffices to consider the torsion-free case. Indeed, assume that $h$ and $h^*$ coincide on endomorphisms of torsion-free abelian groups. 
One has  $h^*(\phi)=h^*(\phi\restriction_{t(G)})+h^*(\overline\phi)$ by (c) and $h(\phi)=h(\phi\restriction_{t(G)})+h(\overline\phi)$ by the Addition Theorem  \ref{AT}. Since $G/t(G)$ is torsion-free, our hypothesis implies $h^*(\overline\phi)=h(\overline\phi)$. Since $h^*(\phi\restriction_{t(G)})=h(\phi\restriction_{t(G)})$ by (i), we conclude that $h^*(\phi)=h(\phi)$.

\smallskip
(iii) If $G$ is torsion-free and $r(G)=n$ is finite, then $h^*(\phi)=h(\phi)$.
Indeed, $D(G)\cong\Q^n$ and $D(G)/G$ is torsion. Let $\overline\phi\in \End(D(G)/G)$ be the endomorphism induced by $\widetilde\phi:D(G)\to D(G)$. By the Algebraic Yuzvinski Formula \ref{Yuz} and by (e) $h^*(\widetilde\phi)=h(\widetilde\phi)$, and this value is finite. Since $D(G)/G$ is torsion, $h^*(\overline\phi)=h(\overline\phi)$ by (i). By (c) $h^*(\widetilde\phi)=h^*(\phi)+h^*(\overline\phi)$ and by the Addition Theorem \ref{AT} $h(\widetilde\phi)=h(\phi)+h(\overline\phi)$. Then $h^*(\phi)=h^*(\widetilde\phi)-h^*(\overline\phi)=h(\widetilde\phi)-h(\overline\phi)=h(\phi)$.

\smallskip
(iv) If $G$ is torsion-free, $G=V(\phi,g)$ for some $g\in G$ and $r(G)$ is infinite, then $h^*(\phi)=\infty=h(\phi)$. The second equality follows from Lemma \ref{CorollaryNov8:1}.
To check the first one, note that $G\cong\bigoplus_{n\in\N}\hull{\phi^n(g)}\cong\Z^{(\N)}$ and $\phi$ is conjugated to $\beta_\Z$ through this isomorphism.
Hence, $h^*(\phi)=\infty$, because $h^*(\phi)=h^*(\beta_\Z)$ by (a) and $h^*(\beta_\Z)=\infty$ by (B).

\smallskip
(v) If $G$ is torsion-free and $G=V(\phi,F)$ for some $F\in[G]^{<\omega}$, then $h^*(\phi)=h(\phi)$.
To prove this, let $F=\{f_1,\ldots,f_k\}$. Then $G=V(\phi,f_1)+\ldots+V(\phi,f_k)$.
If $r(V(\phi,f_i))$ is finite for every $i\in\{1,\ldots,k\}$, then $r(G)$ is finite as well, and $h^*(\phi)=h(\phi)$ by (iii).
If $r(V(\phi,f_i))$ is infinite for some $i\in\{1,\ldots,k\}$, then $h^*(\phi\restriction_{V(\phi,f_i)})=\infty=h(\phi\restriction_{V(\phi,f_i)})$ by (iv). By (A) and by Lemma \ref{restriction_quotient}, $h^*(\phi)=\infty=h(\phi)$.

\smallskip
Consider now the general case. By (ii) we can suppose without loss of generality that $G$ is torsion-free. By Lemma \ref{remark-G_F-gen}, $G=\varinjlim\{V(\phi,F):F\in[G]^{<\omega}\}$, and by (v), $h^*(\phi\restriction_{V(\phi,F)})=h(\phi\restriction_{V(\phi,F)})$ for every $F\in[G]^{<\omega}$. Therefore, (b) and Proposition \ref{properties}(c) give $h^*(\phi)=\sup_{F\in[G]^{<\omega}}h^*(\phi\restriction_{V(\phi,F)})=\sup_{F\in[G]^{<\omega}}h(\phi\restriction_{V(\phi,F)})=h(\phi)$.
\end{proof}

Note that the Uniqueness Theorem \cite[Theorem 6.1]{DGSZ} for $\ent$ in the torsion case uses one more axiom that is not present in our list (a)--(e), namely the Logarithmic Law. As shown in \cite{S}, in the torsion case it follows from (a)--(d). We have seen in the Uniqueness Theorem for $h$ that also in the general case the Logarithmic Law is not among the properties necessary to give uniqueness of $h$ in the category of all abelian groups, that is, the Logarithmic Law follows automatically from (a)--(e).

\bigskip
It is possible to prove the Uniqueness Theorem also in a less direct way, that is, using a known theorem by V\'amos \cite{V} on length functions. We expose this alternative proof in the remaining part of the section, considering the algebraic entropy as a function $h:\Mod_{\Z[t]}\to \R_{\geq0}\cup\{\infty\}$ (see Remark \ref{af-mod}).

\smallskip
Let $R$ be a unitary commutative ring, and consider the category $\Mod_R$ of all $R$-modules and their homomorphisms.
An \emph{invariant} $i:\Mod_R\to \R_{\geq0}\cup\{\infty\}$ is a function such that $i(0)=0$ and $i(M)=i(M')$ if $M\cong M'$ in $\Mod_R$. 
For $M\in\Mod_R$, denote by $\mathcal F(M)$ the family of all finitely generated submodules of $M$.

\begin{definition}\emph{\cite{NR,V}}\label{L-def}
Let $R$ be a unitary commutative ring. A \emph{length function} $L$ of $\Mod_R$ is an invariant $L:\Mod_R\to \R_{\geq0}\cup\{\infty\}$ such that:
\begin{itemize}
\item[(a)] $L(M)=L(M')+L(M'')$ for every exact sequence $0\to M'\to M\to M''\to 0$ in $\Mod_R$;
\item[(b)] $L(M)=\sup\{L(F):F\in\mathcal F(M)\}$.
\end{itemize}
\end{definition}

An invariant satisfying (a) is said \emph{additive} and an invariant with the property in (b) is called \emph{upper continuous}.
So a length function is an additive upper continuous invariant of $\Mod_R$.

\begin{proposition}\label{h-lf} 
If $h^*=\{h^*_G:G\ \text{abelian group}\}$ is a collection of functions $h^*_G:\End(G)\to\R_{\geq0}\cup\{\infty\}$ satisfying (a)--(e) of Theorem \ref{UT}, then $h^*$ is a length function of $\Mod_{\Z[t]}$. In particular, $h$ is a length function of $\Mod_{\Z[t]}$.
\end{proposition}
\begin{proof} 
In terms of Definition \ref{L-def}, $h^*$ and $h$ are upper continuous and additive invariants of $\Mod_{\Z[t]}$, respectively in view of the properties (a), (b) and (c) of Theorem \ref{UT}, and by Example \ref{h(id)=0}, the Addition Theorem \ref{AT} and Proposition \ref{properties}(c).
\end{proof}

As noted in \cite{V}, the values of a length function $L$ of $\Mod_R$ are determined by its values on the finitely generated $R$-modules. 
In case $R$ is a Noetherian commutative ring, $L$ is determined by its values $L(R/\mathfrak p)$ for prime ideals $\mathfrak p$ of $R$ (see \cite[Corollary of Lemma 2]{V}).

\begin{proof}[\bf Second proof of Theorem \ref{UT}]
Let $h^*=\{h^*_G:G\ \text{abelian group}\}$ be a collection of functions $h_G^*:\End(G)\to\R_{\geq0}$ satisfying (a)--(e) of Theorem \ref{UT}. 
Let $R=\Z[t]$. By Proposition \ref{h-lf}, $h^*$ is a length function of $\Mod_R$. In view of the above mentioned results from \cite{V}, it suffices only to check that 
\begin{center}
$h^*(R/\mathfrak p)= h(R/\mathfrak p)$ for all prime ideals $\mathfrak p$ of $R$.
\end{center}

Let us recall that $R$ has Krull dimension $2$, so the non-zero prime ideals $\mathfrak p$ of $R$ are either minimal of maximal. In particular, if $\mathfrak p$ is a minimal prime ideal of $R$, then $\mathfrak p=(f(t))$, where $f(t)\in R$ is irreducible (either $f(t)=p$ is a prime in $\Z$, or $f(t)$ is irreducible with $\deg f(t)>0$). On the other hand, a maximal ideal $\mathfrak m$ of $R$ is of the form $\mathfrak m=(p,f(t))$, where $p$ is a prime in $\Z$ and $f(t)\in R$ 
 is irreducible modulo $p$ and $\deg f(t)>0$.

\smallskip
(i) For $\mathfrak p=0$, we prove that $h^*(R)=\infty=h(R)$. To this end we have to show that $\mu_t:R\to R$ has $h^*(\mu_t)=\infty=h(\mu_t)$. The flow $(R,\mu_t)$ is isomorphic to the flow $(\Z^{(\N)},\beta_\Z)$. By (a) and by Proposition \ref{properties}(a) respectively, this yields $h^*(\mu_t)=h^*(\beta_\Z)$ and $h(\mu_t)=h(\beta_\Z)$. As $h^*(\beta_\Z)=\infty=h(\beta_\Z)$ by (B) and Example \ref{beta}, we are done.

\smallskip
(ii) To see that $h^*(R/\mathfrak m)=h(R/\mathfrak m)$ for a maximal ideal $\mathfrak p=\mathfrak m$ of $R$, it suffices to show that $R/\mathfrak m$ is finite and apply the Uniqueness Theorem for the torsion case \cite[Theorem 6.1]{DGSZ}.
Indeed, $\mathfrak m=(p,f(t))$, where $p\in\Z$ is a prime and $f(t)\in R$ is irreducible modulo $p$. So, $R/\mathfrak m\cong(\Z/p\Z)[t]/(f_p(t))$, where $f_p(t)$ is the reduction of $f(t)$ modulo $p$. Hence, $R/\mathfrak m$ is finite. 

\smallskip
(iii) It remains to see that  $h^*(R/\mathfrak p)=h(R/\mathfrak p)$ when $\mathfrak p$ is a minimal prime ideal of $R$. 

\smallskip
Assume first that $\mathfrak p=(p)$ for some prime $p\in\Z$. We show that $h^*(R/\mathfrak p)=\log p=h(R/\mathfrak p)$.

Indeed, $R/(p)\cong(\Z/p\Z)[t]$ and $(\Z/p\Z)[t]\cong\Z(p)^{(\N)}$ as abelian groups. Moreover, $\mu_t:(\Z/p\Z)[t]\to (\Z/p\Z)[t]$ is conjugated to $\beta_{\Z(p)}$ through this isomorphism. By (a), (B) and Proposition \ref{properties}(a), we get the required equality.  

\smallskip
Suppose now that $\mathfrak p=(f(t))$, where $f(t)=a_0+a_1t+\ldots+a_nt^n\in \Z[t]$ is irreducible (so, primitive) with $\deg f(t)=n>0$. 
Let $M=R/\mathfrak p$. We verify that $h^*(M)=h(M)$.

Let $J$ be the principal ideal generated by $f(t)$ in $\Q[t]$ and $D=\Q[t]/J$. Let $\pi:\Q[t]\to D$ be the canonical projection. Since $J\cap R=\mathfrak p$, $\pi$ induces an injective homomorphism $M\to D$ and we can think without loss of generality that $M$ is a subgroup of $D$ (identifying $M$ with $\pi(R)$).
Since $D\cong\Q^n$ as abelian groups and $r(M)\geq n$, $M$ is essential in $D$. Consider $\mu_t:D\to D$ and $\mu_t\restriction_M:M\to M$.
Since $D/M$ is torsion, $h^*(\overline\mu_t)=h(\overline\mu_t)=0$ by the Uniqueness Theorem for the torsion case \cite[Theorem 6.1]{DGSZ}. 
By (C), $h^*(\mu_t) =m(f(t))= h(\mu_t)< 	\infty$. Moreover, $h^*(\mu_t\restriction_M) = h^*(\mu_t) - h^*(\overline\mu_t)$, by (c); while $h(\mu_t\restriction_M) = h(\mu_t) - h(\overline\mu_t)$ by the Addition Theorem \ref{AT}. This yields $h^*(\mu_t\restriction_M) = h(\mu_t\restriction_M)$, i.e., $h^*(M)=h(M)$.
\end{proof}

\section{Entropy vs Mahler measure}\label{mahler-sec}

\subsection{Computation of entropy via Mahler measure}

For an algebraic number $\alpha\in\C$, the \emph{Mahler measure} $m(\alpha)$ of $\alpha$ is the Mahler measure of $sf(t)$, where $f(t)$ is the minimal polynomial of $\alpha$ over $\Q$ and $s$ is the least positive common multiple of the denominators of $f$.
In order to characterize the algebraic numbers $\alpha$ with $m(\alpha)=0$ one need the following:

\begin{theorem}[Kronecker Theorem]\label{Kr}\emph{\cite{Kr}}
Let $f(t)\in\Z[t]$ be a monic polynomial with roots $\alpha_1, \ldots, \alpha_k$. If $\alpha_1$ is not a root of unity, then $|\alpha_i|>1$ for at least one $i\in\{1,\ldots,k\}$. 
\end{theorem}

\begin{corollary}\label{Kronecker}
Let $f(t)\in\Z[t]\setminus\{0\}$ be a primitive polynomial. Then $m(f(t))=0$ if and only if $f(t)$ is cyclotomic (i.e., all the roots of $f(t)$ are roots of unity).
Consequently, if $\alpha$ is an algebraic integer, then $m(\alpha)=0$ if and only if $\alpha$ is a root of unity.
\end{corollary}
\begin{proof}
Let $f(t)=a_0+a_1t+\ldots+a_kt^k$ and $\deg f(t)=k$. Assume that all roots $\alpha_1,\ldots,\alpha_k$ of $f(t)$ are roots of unity. Then there exists $m\in\N_+$ such that every $\alpha_i$ is a root of $t^m-1$. By Gauss Lemma $f(t)$ divides $t^m-1$ in $\Z[t]$. Therefore, $a_k=\pm 1$, and hence $m(f(t))=0$.
Suppose now that $m(f(t))=0$. This entails  that $f(t)$ is monic. Moreover, $|\alpha_i|\leq 1$ for every $i\in\{1,\ldots,k\}$. By Kronecker Theorem \ref{Kr} each $\alpha_i$ is a root of unity.
\end{proof}

This determines completely the case of zero Mahler measure. For the greatest lower bound of the positive values of the Mahler measure one has Lehmer Problem \ref{h>0-pb}, that can be reformulated also as follows: 

\begin{problem}
Is  $\inf\{m(\alpha)>0:\alpha\ \text{algebraic integer}\}>0$?
\end{problem}
 
Obviously, the problem can be formulated equivalently for algebraic numbers, as $m(\alpha)\ge \log 2$ for every algebraic number $\alpha$ that is not an algebraic integer. 

\smallskip
We prove now that $\mathfrak L=\varepsilon$, that is, that the infimum of the positive values of the Mahler measure coincides with the infimum of the positive values of the algebraic entropy.

\begin{proof}[\bf Proof of Theorem \ref{epsilon}]
We start by proving the first equality in Theorem \ref{epsilon}, namely 
\begin{equation}\label{=1}
\varepsilon=\inf\{h(\phi)>0:n\in\N_+,\phi\in\Aut(\Q^n)\}. 
\end{equation}
We split the proof in three steps of reductions. 

(i) Let $G$ be an abelian group, $\phi\in \End(G)$ and $\overline\phi\in \End(G/t(G))$. Since $h(\phi\restriction_{t(G)})=\ent(\phi\restriction_{t(G)})$, we can suppose that $h(\phi\restriction_{t(G)})=0$, otherwise $h(\phi)\geq h(\phi\restriction_{t(G)})\geq\log2$ by Lemma \ref{restriction_quotient} and \eqref{ent>log2}. By the Addition Theorem \ref{AT} $h(\phi)=h(\overline\phi)$. In other words, we can consider only torsion-free abelian groups $G$.

\smallskip
(ii) By (i), we can assume that $G$ is a torsion-free abelian group and $\phi\in\End(G)$. If there exists $g\in G$ such that $r(V(\phi,g))$ is infinite, then $h(\phi)\geq h(\phi\restriction_{V(\phi,g)})=\infty$ by Lemma \ref{restriction_quotient} and Corollary \ref{cyc_case}(b). Hence, we can assume that $r(V(\phi,g))$ is finite for every $g\in G$. Then $r(V(\phi,F))$ is finite for every $F\in[G]^{<\omega}$.
By Lemma \ref{remark-G_F-gen}, $h(\phi)=\sup_{F\in[G]^{<\omega}}h(\phi\restriction_{V(\phi,F)})$. So, we can consider only torsion-free abelian groups of finite rank. Furthermore, $h(\phi)=h(\widetilde\phi)$ by Theorem \ref{AA}, and hence we can reduce to divisible torsion-free abelian groups of finite rank.

\smallskip
(iii) By (ii) and by Proposition \ref{properties}(a), we can assume that $G=\Q^n$ for some $n\in\N_+$. Let $\phi\in\End(\Q^n)$. By Proposition \ref{red-to-inj} $h(\phi)=h(\overline\phi)$, where the induced endomorphism $\overline\phi\in \End(\Q^n/\ker_\infty\phi)$ is injective (hence, surjective). Moreover, $\Q^n/\ker_\infty\phi\cong\Q^m$ for some $m\in\N$, $m\leq n$, as $\ker_\infty\phi$ is pure in $\Q^n$ (so $\Q^n/\ker_\infty\phi$ is divisible and torsion-free by Lemma \ref{div<->pure}(c)). Therefore, $\overline\phi\in \Aut(\Q^m)$, and this proves the equality in \eqref{=1}.

\smallskip
The equality $\mathfrak L=\varepsilon$ follows from \eqref{=1} and the Algebraic Yuzvinski Formula \ref{Yuz}.
\end{proof}

Following \cite{DG_Pak}, for an abelian group $G$, let ${\bf E}_{alg}(G) = \{h(\f) : \f \in \End(G)\}$ and let ${\bf E}_{alg}$ be the union of all ${\bf E}_{alg}(G)$ when $G$ varies among the class of all abelian groups. Then ${\bf E}_{alg}$ is a submonoid of the monoid $(\mathbb R_{\geq0}\cup\{\infty\},+,0)$, and $\mathfrak L=\inf ({\bf E}_{alg}\setminus \{0\})$ by Theorem \ref{epsilon}.
Moreover, let  
$$\mathbb A=\{\log|\alpha|: \alpha\ \text{algebraic number},\ |\alpha|\geq1\},$$ 
which is a dense countable submonoid of $(\R_{\geq0},+,0)$. The next theorem collects some equivalent forms of $\mathfrak L=0$ (a  proof can be found in \cite{DG_Pak}). 

\begin{theorem}\label{LD}
The following conditions are equivalent:
\begin{itemize}
  \item[(a)] ${\bf E}_{alg}=\R_{\geq0}\cup \{\infty\}$;
  \item[(b)] $\mathfrak L=0$; 
  \item[(c)] $\inf(\bigcup_{n\in\N_+}{\bf E}_{alg}(\Z^n)\setminus \{0\})=0$.
\end{itemize}
If $\mathfrak L>0$, then ${\bf E}_{alg}\subseteq \mathbb A\cup\{\infty\}$, so in particular it is countable. 
\end{theorem}

The inclusion ${\bf E}_{alg}\subseteq \mathbb A\cup\{\infty\}$ is proper in case $\mathfrak L>0$ by the density of the set $\mathbb A$ in $\R_{\geq0}$. 

\begin{remark}
By the Algebraic Yuzvinski Formula \ref{Yuz}, the algebraic entropy $h(\phi)$ of an endomorphism $\phi:\Q^n\to \Q^n$, $n\in\N_+$, is equal to the Mahler measure of a polynomial $f(t)\in\Z[t]$,  namely $f(t)=sg(t)\in\Z[t]$, where $g(t)\in\Q[t]$ is the characteristic polynomial of $\phi$ and $s$ is the least positive common multiple of the denominators of $g(t)$. One can read this in the opposite direction, that is, the Mahler measure of a polynomial $f(t)\in\Z[t]$ of degree $n$ is equal to the algebraic entropy of some endomorphism $\phi$ of $\Q^n$. Indeed, let $f(t)=a_0+a_1t+\ldots+a_nt^n\in\Z[t]$ with $\deg f(t)=n$ and let 
\begin{equation}\label{companion}
C(f)=\begin{pmatrix}
0 & 0 & \ldots & 0 & -\frac{a_0}{a_n}\\
1 & 0 & \ldots & 0 & -\frac{a_1}{a_n}\\
0 & 1 & \ddots & 0 & -\frac{a_2}{a_n}\\
\vdots & \ddots & \ddots & \vdots & \vdots\\
0 & 0 & \ldots & 1 & -\frac{a_{n-1}}{a_n}
\end{pmatrix}
\end{equation}
be the companion matrix associated to $f(t)$. The characteristic polynomial of $C(f)$ is $f(t)$, and $C(f)$ is associated to an endomorphism $\phi$ of $\Q^n$ (see also Remark \ref{LaastRemark} about this endomorphism). Then $m(f(t))=h(\phi)$ by the Algebraic Yuzvinski Formula \ref{Yuz}.
\end{remark}

Finally, we present the above situation from a different point of view, namely the \emph{algebraic entropy} associated to the multiplication by an element in a commutative ring, defined in \cite{Z}.

\begin{remark}\label{LaastRemark}
Let $\alpha\in\mathbb C$ be an algebraic number of degree $n\in\N_+$ over $\Q$ and let $f(t)\in\Q[t]$ be its minimal polynomial, with $\deg f(t)=n$.
Let $K=\Q(\alpha)$ be the simple field extension of $\Q$ associated to $\alpha$; then $K\cong\Q^n$ as abelian groups. Inspired by \cite{Z}, call \emph{algebraic entropy} $h(\alpha)$ of $\alpha$ the algebraic entropy $h(\mu_\alpha)$, where $\mu_\alpha$ is the multiplication by $\alpha$ in $K$. 
Let $a_n$ be the smallest positive integer such that $a_n f(t)=a_0+a_1t+\ldots+a_nt^n\in\Z[t]$. With respect to the basis $\{1,\alpha,\ldots,\alpha^{n-1}\}$ of $K$, the matrix associated to $\mu_\alpha$ is the companion matrix $C(f)$ in \eqref{companion}.
The characteristic polynomial of $C(f)$ is $f(t)$. By the Algebraic Yuzvinski Formula \ref{Yuz}, $h(\mu_\alpha)=m(f(t))$, i.e., $h(\alpha)=m(\alpha)$.
This means that the algebraic entropy of an algebraic number coincides with its Mahler measure.
\end{remark}

\subsection{Lehmer Problem and realization of entropy}

This section is dedicated to the proof of the Realization Theorem, showing that Lehmer Problem is equivalent to the ``realization problem'' of the finite values of the algebraic entropy.

\medskip
In the next lemma we collect known results showing the validity of the Realization Theorem for finitely generated flows satisfying some additional restraint. 

\begin{lemma}\label{realtor}
Let $G$ be an abelian group, $\f\in\End(G)$ and assume that $G$ is finitely $\phi$-generated by $F\in[G]^{<\omega}$. 
\begin{itemize}
\item[(a)]\emph{\cite{DGSZ}} If $G$ is torsion, then $h(\phi)=H(\phi,F)$.
\item[(b)]\emph{\cite{DG}} If $H(\phi,F)=0$, then $h(\phi)=0$ (in particular, $h(\phi)=H(\phi,F))$.
\end{itemize}
\end{lemma}

The conclusion $h(\phi)=H(\phi,F)$ in the above lemma does not hold true in general. Indeed, consider $\mu_3 :\Z \to\Z$ and $F=\{0,1\}$; now $H(\mu_3,F) \leq \log |F| = \log 2$, while $h(\mu_3) = \log 3$. This can be generalized to any non-torsion infinite finitely generated abelian group $G$ with $F$ a finite set of generators of $G$; in this case, for every $k\in\N_+$ with $k>|F|$, one has $h(\mu_k)\geq\log k>\log|F|$.

\medskip
The following is a direct consequence of Proposition \ref{AA}.

\begin{corollary}\label{realdiv}
Let $G$ be a torsion-free abelian group and $\f\in\End(G)$ with $h(\phi)<\infty$. Then $h(\widetilde\phi)$ realizes if and only if $h(\phi)$ realizes.
\end{corollary}
\begin{proof}
If $h(\phi)$ realizes, since $h(\phi)=h(\widetilde\phi)$ by Proposition \ref{AA}(b), then also $h(\widetilde\phi)$ realizes.
Assume now that $h(\widetilde\phi)$ realizes. Then there exists $F\in[D(G)]^{<\omega}$ such that $h(\widetilde\phi)=H(\widetilde\phi,F)$. By Proposition \ref{AA} there exists $m\in\N_+$ such that $mF\subseteq G$ and $h(\phi)=h(\widetilde\phi)=H(\widetilde\phi,F)=H(\phi,mF)$, therefore, $h(\phi)$ realizes.
\end{proof}

The following technical lemma is needed in the proof of Proposition \ref{realttf}.

\begin{lemma}\label{realVn}
Let $G$ be an abelian group and $\phi\in\End(G)$ with $h(\phi)<\infty$. Let $\{V_n\}_{n\in\N}$ be an increasing chain of $\phi$-invariant subgroups of $G$, such that $V_0=0$ and  $h(\phi)=\sup_{n\in\N}h(\phi\restriction_{V_n})$. Let $\phi_n= \phi\restriction_{V_n}$ and $\overline\phi_{n}\in \End(V_{n}/V_{n-1})$ be the induced endomorphism for $n \in \N_+$. Assume that $\eta>0$ and that $h(\overline\phi_{n})$ is either $0$ or $\geq\eta$ for every $n\in\N_+$. Then $h(\phi)=h(\phi_n)$ for some $n\in\N$. 
\end{lemma}
\begin{proof} 
For the sake of completeness let $\phi_{0}=  \phi\restriction_{V_0}$ and note that $\phi_1 = \overline\phi_{1}$. 
The sequence $\{h(\phi_n)\}_{n\in\N}$ is an increasing sequence of non-negative real numbers.
For every $n\in\N_+$, let 
$$d_{n}=h(\phi_{n})-h(\phi_{n-1}).$$
By the Addition Theorem \ref{AT}, $h(\phi_{n})=h(\phi_{n-1})+h(\overline\phi_{n})$ for every $n\in\N_+$. Hence, 
$d_n=h(\overline\phi_{n})$, and so $d_n$ is either $0$ or $\geq\eta$ by hypothesis. Therefore, for every $n\in\N$, 
$$h(\phi_{n})=\sum_{i=1}^{n} d_i\geq m_{n}\eta,\ \text{where}\ m_{n}=|\{d_i\neq0:1\leq i \leq n\}|.$$
Since $h(\phi_{n})\leq h(\phi)$ for every $n\in\N$, we have that 
$$m_{n}\leq \frac{h(\phi)}{\eta}.$$ 
Consequently, there are finitely many $d_n\neq 0$, and this means that the sequence $\{h(\phi_n)\}_{n\in\N}$ stabilizes. By hypothesis, $h(\phi)=\sup_{n\in\N}h(\phi_n)$, hence $h(\phi)=h(\phi_n)$ for every sufficiently large $n\in\N$.
\end{proof}

From Lemma \ref{realVn} and Corollary \ref{realdiv} we deduce respectively item (a) and item (b) of the following result.

\begin{proposition}\label{realttf}
Let $G$ be an abelian group and $\f\in\End(G)$ with $h(\phi)<\infty$. Then $h(\phi)$ realizes when:
\begin{itemize}
  \item[(a)] $G$ is torsion;
  \item[(b)] $G$ is torsion-free of finite rank.
\end{itemize}
\end{proposition}

\begin{proof}
(a)  By Proposition \ref{sup->lim} there exists a chain $\{V_n\}_{n\in\N}$ in $\mathfrak F(G,\phi)$ such that  $h(\phi)=\lim_{n\to\infty}h(\phi\restriction_{V_n})$. We can apply Lemma \ref{realVn} with $\eta = \log 2$ to find $n\in\N$ such that $h(\phi)=h(\phi\restriction_{V_n})$. We have $V_n=V(\phi,F)$ for some $F\in [G]^{<\omega}$. Since $G$ is torsion, $F'=\langle F\rangle$ is finite and $V_n=V(\phi,F')=T(\phi,F')$. By Lemma \ref{realtor}(a), $h(\phi\restriction_{V_n})=H(\phi,F')$, and hence $h(\phi)=h(\phi\restriction_{V_n}) = H(\phi,F')$.

\smallskip
(b) If $G$ is a torsion-free abelian group of finite rank, by Corollary \ref{realdiv} we can assume that $G$ is also divisible and then $G\cong\Q^n$ for some $n\in\N_+$. By \cite[Theorem 4.18]{GV} there exist infinitely many $F\in[G]^{<\omega}$ such that $h(\phi)=H(\phi,F)$.
\end{proof}

We are now in position to prove the Realization Theorem for the algebraic entropy:

\begin{proof}[\bf Proof of Theorem \ref{RT}]
(b)$\Rightarrow$(a) Assume that $\mathfrak L=0$. By Theorem \ref{LD} this is equivalent to ${\bf E}_{alg}=\R_{\geq0}\cup \{\infty\}$. Hence, for every $n\in\N$ there exist an abelian group $G_n$ and $\phi_n\in\End(G_n)$ such that $h(\phi_n)=\frac{1}{2^n}$. Now let $G=\bigoplus_{n\in\N}G_n$ and $\phi=\bigoplus_{n\in\N}\phi_n$. Then $h(\phi)=\sum_{n\in\N}\frac{1}{2^n}=1$ by Proposition \ref{properties}(c,d). To conclude it suffices to note that for every $F\in[G]^{<\omega}$ there exists $n\in\N$ such that $F\subseteq G_0\oplus\ldots\oplus G_n$, and so $H(\phi,F)\leq h(\phi_0\oplus\ldots\oplus\phi_n)=\sum_{i=0}^n\frac{1}{2^i}<1$ by Proposition \ref{properties}(d).  Then $h(\phi)=1$, but $h(\phi)$ does not realize.

\smallskip
(a)$\Rightarrow$(b) By Proposition \ref{sup->lim} there exists a chain $\{V_n\}_{n\in\N}$ in $\mathfrak F(G,\phi)$ such that $h(\phi)=\lim_{n\to\infty}h(\phi\restriction_{V_n})$. As $\mathfrak L>0$, by Lemma \ref{realVn} (with $\eta = \mathfrak L$) there exists $n\in\N$ such that $h(\phi)=h(\phi\restriction_{V_n})$. Hence, we can assume without loss of generality that $G$ is finitely $\phi$-generated by some $F\in[G]^{<\omega}$.

Since $h(\phi)<\infty$, Lemma \ref{X} implies that $V(\phi,g)$ has finite rank for every $g\in G$. 
Then $G$ has finite rank as well, and by Lemma \ref{fin-gen+noet->fin-gen}(b) there exists a finite-rank torsion-free subgroup $K$ of $G$ such that $G\cong K\times t(G)$. By Example \ref{skewex}, $\phi$ is the skew product of $\phi_1$ and $\phi_2$, where $\f_1:K\to K$ is conjugated to $\overline\f:G/t(G)\to G/t(G)$ by the isomorphism $K\cong G/t(G)$, and $\f_2=\f\restriction_{t(G)}$; moreover, $s_\f(K)$ is finite.
By Proposition \ref{realttf}(a) there exists $F_2\in [t(G)]^{<\omega}$ such that $h(\phi_2)=H(\phi_2,F_2)$. We can assume without loss of generality that $s_\phi(K)\subseteq F_2$. Since $K\cong G/t(G)$ is torsion-free of finite rank, by Proposition \ref{realttf}(b) there exists $F_1\in[K]^{<\omega}$ such that $h(\phi_1)=H(\phi_1,F_1)$. 
By Proposition \ref{properties}(d) and Proposition \ref{AT-semi-poor}, for $\pi_\phi=\phi_1\times\phi_2$ one has 
$$h(\phi)=h(\pi_\phi)=h(\phi_1)+h(\phi_2)=H(\phi_1,F_1)+H(\phi_2,F_2)=H(\phi_1\times\phi_2,F_1\times F_2)=H(\phi,F_1\times F_2),$$
as $T_n(\pi_\phi,F_1\times F_2)=T_n(\phi,F_1\times F_2)$ for every $n\in\N_+$. This concludes the proof for the case $G=V(\phi,F)$, and so the proof of the theorem.
\end{proof}


\begin{thebibliography}{99}

\bibitem{Ab} L. M. Abramov, \emph{The entropy of an automorphism of a solenoidal group}, Teor. Veroyatnost. i Primenen 4 (3) (1959) 249--254.

\bibitem{AKM} R. L. Adler, A. G. Konheim, M. H. McAndrew, \emph{Topological entropy}, Trans. Amer. Math. Soc. 114 (1965) 309--319.

\bibitem{AADGH} M. Akhavin, F. Ayatollah Zadeh Shirazi, D. Dikranjan, A. Giordano Bruno, A. Hosseini, \emph{Algebraic entropy of shift endomorphisms on abelian groups}, Quaest. Math. 32 (4) (2009) 529--550.

\bibitem{A} D. Arnold, \emph{Finite Rank Torsion Free Abelian Groups and Rings}, L.N.M. n. 931, Springer 1982.

\bibitem{Bow} R. Bowen, \emph{Entropy and the fundamental group}, in: The Structure of Attractors in Dynamical Systems (Proc. Conf. North Dakota State Univ., Frago, 1977), Lecture Notes in Math. 668 (1978) 21--29.

\bibitem{DGS} D. Dikranjan, A. Giordano Bruno, L. Salce, \emph{Adjoint algebraic entropy}, Journal of Algebra {324} (2010) 442--463.

\bibitem{DG-peters} D. Dikranjan, A. Giordano Bruno, \emph{Entropy on abelian groups}, arxiv.org/pdf/1007.0533v1.pdf (2010).

\bibitem{DG} D. Dikranjan, A. Giordano Bruno, \emph{The Pinsker subgroup of an algebraic flow}, Jour. Pure Appl. Algebra 216 (2012) 364--376.

\bibitem{DG-lf} D. Dikranjan, A. Giordano Bruno, \emph{Limit free computation of entropy}, Rend. Istit. Mat. Univ. Trieste 44 (2012) 297--312.

\bibitem{DG2} D. Dikranjan, A. Giordano Bruno, \emph{The connection between topological and algebraic entropy}, Topology Appl. 159 (13) (2012) 2980--2989.

\bibitem{DG_Pak} D. Dikranjan, A. Giordano Bruno, \emph{Topological and algebraic entropy for group endomorphisms}, Proceedings ICTA2011 Islamabad, Pakistan July 4-10, 2011 Cambridge Scientific Publishers (2012) 133--214.

\bibitem{DG1} D. Dikranjan, A. Giordano Bruno, \emph{Entropy in a category}, Appl. Categ. Structures 21 (1) (2013) 67--101.

\bibitem{DG_PC} D. Dikranjan, A. Giordano Bruno, \emph{Discrete dynamical systems in group theory}, Note Mat. 33 (1) (2013) 1-48.

\bibitem{DG-BT} D. Dikranjan, A. Giordano Bruno, \emph{The Bridge Theorem for totally disconnected LCA groups}, Topology Appl. 169 (1) (2014) 21--32.

\bibitem{DGSV} D. Dikranjan, A. Giordano Bruno, L. Salce, S. Virili, \emph{Intrinsic algebraic entropy}, J. Pure Appl. Algebra 219 (2015) 2933--2961.

\bibitem{DGSZ} D. Dikranjan, B. Goldsmith, L. Salce, P. Zanardo, \emph{Algebraic entropy of endomorphisms of abelian groups}, Trans. Amer. Math. Soc. {361} (2009) 3401--3434.

\bibitem{DGZ} D. Dikranjan, K. Gong, P. Zanardo, \emph{Endomorphisms of abelian groups with small algebraic entropy}, Linear Algebra Appl. 439 (7) (2013) 1894--1904. 

\bibitem{DPS} D. Dikranjan, I. Prodanov, L. Stoyanov, \emph{Topological Groups: Characters,  Dualities  and  Minimal Group Topologies}, Pure and Applied Mathematics, Vol. 130, Marcel Dekker Inc., New York-Basel, 1989.

\bibitem{DSV} D. Dikranjan, M. Sanchis, S. Virili, \emph{New and old facts about entropy on uniform spaces and topological groups}, Topology Appl. 159 (7) 
(2012) 1916--1942.

\bibitem{Ward0}  G. Everest, T. Ward, \emph{Heights of polynomials and entropy in algebraic dynamics}, Universitext. Springer-Verlag London Ltd., London 1999. 

\bibitem{FFK} K. Falconer, B. Fine, D. Kahrobaei, \emph{Growth rate of an endomorphism of a group}, Groups Complex. Cryptol. 3 (2011) no. 2 285--300. 

\bibitem{Fek} M. Fekete, \emph{\"Uber die Verteilung der Wurzeln bei gewisser algebraichen Gleichungen mit ganzzahlingen Koeffizienten}, Math. Zeitschr. 17 (1923) 228--249.

\bibitem{F} L. Fuchs, \emph{Infinite Abelian Groups}, Vol. I and II, Academic Press, 1970 and 1973.

\bibitem{GB} A. Giordano Bruno, \emph{Algebraic entropy of shift endomorphisms on products}, Comm. Algebra 38 (11) (2010) 4155--4174.

\bibitem{GV} A. Giordano Bruno, S. Virili, \emph{The Algebraic Yuzvinski Formula}, J. Algebra 423 (2015) 114--147. 

\bibitem{GV-app} A. Giordano Bruno, S. Virili, \emph{About the Algebraic Yuzvinski Formula}, Topol. Algebra and its Appl. 3 (1) (2015) 86--103.

\bibitem{Halmos2} P. Halmos, \emph{Finite-Dimensional Vector Spaces}, Undergraduate Texts in Mathematics, Springer-Verlag, New York - Heidelberg - Berlin, 1987.

\bibitem{HR}  E. Hewitt, K. A. Ross, \emph{Abstract harmonic analysis I}, Springer-Verlag, Berlin-Heidelberg-New York, 1963.

\bibitem{Hi} E. Hironaka, \emph{What is\ldots Lehmer's number?}, Not. Amer. Math. Soc. {56} (3) (2009) 374--375.


\bibitem{Kr}  L. Kronecker, \emph{Zwei S\" atze \" uber Gleichungen mit ganzzahligen Coefficienten}, Jour. Reine Angew. Math. {53} (1857) 173--175.

\bibitem{K} A. N. Kolmogorov, \emph{New metric invariants of transitive dynamical systems and automorphisms of Lebesgue spaces}, Doklady Akad. Nauk. SSSR 119 (1958) 861--864 (in Russian).

\bibitem{L} D. H. Lehmer, \emph{Factorization of certain cyclotomic functions}, Ann. of Math. {34} (1933) 461--469.

\bibitem{LW} D. Lind, T. Ward, \emph{Automorphisms of solenoids and $p$-adic entropy}, Ergodic Theory Dynam. Systems 8 (3) (1988) 411--419.

\bibitem{MRQ} M. J. Mossinghoff, G. Rhin, Q. Wu, \emph{Minimal Mahler measures}, Experiment. Math. 17 (4) (2008) 451--458.


\bibitem{NR} D. G. Northcott, M. Reufel, \emph{A generalization of the concept of length}, Quart. J. of Math. (Oxford) (2) 16 (1965) 297--321.

\bibitem{Pet} J. Peters, \emph{Entropy on discrete Abelian groups}, Adv. Math. {33} (1979) 1--13.

\bibitem{Pet1}  J. Peters, \emph{Entropy of automorphisms on {L}.{C}.{A}. groups}, Pacific J. Math. {96} (2) (1981) 475--488.

\bibitem{P} L. S. Pontryagin, \emph{Topological Groups}, Gordon and Breach, New York 1966.

\bibitem{S} L. Salce, \emph{Some results on the algebraic entropy},  Groups and Model Theory, Contemporary Math. 576 (2012) 297--304.

\bibitem{SVV} L. Salce, P. V\'amos, S. Virili, \emph{Length functions, multiplicities and algebraic entropy}, Forum Math. 25 (2) (2013) 255--282.

\bibitem{SZ2} L. Salce, P. Zanardo, \emph{A general notion of algebraic entropy and the rank entropy},  Forum Math. 21 (4) (2009) 579--599.

\bibitem{Sinai} Y. G. Sinai, \emph{On the concept of entropy of a dynamical system}, Doklady Akad. Nauk. SSSR 124 (1959) 786--781 (in Russian).

\bibitem{Sm} C. Smyth, \emph{On the product of conjugates outside the unit circle of an algebraic integer}, Bull. London Math. Soc. {3} (1971) 169--175.
         
\bibitem{St} L. N. Stoyanov, \emph{Uniqueness of topological entropy for endomorphisms on compact groups}, Boll. Un. Mat. Ital. B (7) 1 (3) (1987) 829--847.

\bibitem{V} P. V\'amos, \emph{Additive Functions and Duality over Noetherian Rings}, Quart. J. of Math. (Oxford) (2) {19} (1968) 43--55.

\bibitem{Vi} S. Virili, \emph{Entropy for endomorphisms of LCA groups}, Topology Appl. 159 (9) (2012) 2546--2556.

\bibitem{Vi1} S. Virili, \emph{Algebraic and topological entropy of group actions}, preprint.


\bibitem{W}  M. D. Weiss, \emph{Algebraic and other entropies of group endomorphisms}, Math. Systems Theory {8} (3) (1974/75) 243--248.

\bibitem{Y} S. Yuzvinski, \emph{Metric properties of endomorphisms of compact groups}, Izv. Acad. Nauk SSSR, Ser. Mat. {29} (1965) 1295--1328 (in Russian). English Translation: Amer. Math. Soc. Transl. (2) 66 (1968) 63--98.

\bibitem{juz67} S.~A. Yuzvinski, \emph{Calculation of the entropy of a group-endomorphism}, Sibirsk. Mat. \u Z. 8 (1967) 230--239.

\bibitem{Z} P. Zanardo, \emph{Algebraic entropy of endomorphisms over local one-dimensional domains}, J. Algebra Appl. {8} (6) (2009) 759--777.

\end{thebibliography}
\end{document}